\documentclass[12pt]{scrartcl}
\usepackage{a4}
\usepackage{amsthm}
\usepackage{amsmath}
\usepackage{amssymb}
\usepackage{amsfonts}
\usepackage{mathrsfs}
\usepackage{mathtools}
\usepackage{extarrows}
\usepackage{dsfont}
\usepackage{latexsym}
\usepackage{color}
\usepackage{bbm,exscale}
\usepackage{bm}
\definecolor{Myblue}{rgb}{0,0,0.6}
\usepackage[a4paper,colorlinks,citecolor=Myblue,linkcolor=Myblue,urlcolor=Myblue,pdfpagemode=None]{hyperref}
\usepackage[square,numbers,sort&compress]{natbib}
\usepackage[all,cmtip]{xy}
\usepackage{graphicx}
\usepackage{tikz}
\usetikzlibrary{decorations.pathmorphing}
\usetikzlibrary{calc}
\usetikzlibrary{decorations.markings}
\usetikzlibrary{fadings,decorations.pathreplacing}
\usetikzlibrary{matrix,arrows}
\usetikzlibrary{arrows,calc,decorations.pathreplacing,decorations.markings,shapes.geometric,shadows}

  \vfuzz \hfuzz
\DeclareFontFamily{U}{BOONDOX-calo}{\skewchar\font=45 }
\DeclareFontShape{U}{BOONDOX-calo}{m}{n}{
  <-> s*[1.05] BOONDOX-r-calo}{}
\DeclareFontShape{U}{BOONDOX-calo}{b}{n}{
  <-> s*[1.05] BOONDOX-b-calo}{}

\newcommand{\tens}[1]{%
  \mathbin{\mathop{\otimes}\limits^{#1}}%
}

\makeatletter
\newcommand{\raisemath}[1]{\mathpalette{\raisem@th{#1}}}
\newcommand{\raisem@th}[3]{\raisebox{#1}{$#2#3$}}
\makeatother

\newcommand{\E}{\text{e}}
\newcommand{\I}{\text{i}}

\newcommand{\B}{\mathcal{B}}

\newcommand{\Beq}{\B_{\mathrm{eq}}}

\newcommand{\C}{\mathds{C}}

\newcommand{\Q}{\mathds{Q}}

\newcommand{\Z}{\mathds{Z}}

\def\1{\ifmmode\mathrm{1\!l}\else\mbox{\(\mathrm{1\!l}\)}\fi}

\newcommand{\be}{\begin{equation}}
\newcommand{\ee}{\end{equation}}
\newcommand{\bes}{\begin{equation*}}
\newcommand{\ees}{\end{equation*}}

\newcommand{\Hom}{\operatorname{Hom}}
\newcommand{\End}{\operatorname{End}}
\newcommand{\modu}{\operatorname{mod}}
\def\LG{\mathcal{LG}}
\def\LGgr{\mathcal{LG}^{\mathrm{gr}}}

\newcommand{\hmf}{\operatorname{hmf}}
\newcommand{\hmfgr}{\operatorname{hmf}^{\textrm{gr}}}

\newcommand{\ev}{\operatorname{ev}}
\newcommand{\tev}{\widetilde{\operatorname{ev}}}
\newcommand{\coev}{\operatorname{coev}}
\newcommand{\tcoev}{\widetilde{\operatorname{coev}}}
\def\lra{\longrightarrow}
\def\lmt{\longmapsto}

\DeclareMathOperator{\str}{str}

\DeclareMathOperator{\Res}{Res}
\newcommand{\diml}{\dim_{\mathrm{l}}}
\newcommand{\dimr}{\dim_{\mathrm{r}}}

\newcommand*{\longhookrightarrow}{\ensuremath{\lhook\joinrel\relbar\joinrel\rightarrow}}

\def\DGs{\mathcal{DG}^{\textrm{s}}}
\def\DGsp{\mathcal{DG}^{\textrm{sp}}}
\def\DGsat{\mathcal{DG}^{\textrm{sat}}}

\def\op{{\textrm{op}}}

\def\to{\rightarrow}
\def\lto{\longrightarrow}

\newcommand{\ZZ}{\ensuremath{\mathbb{Z}}}

\newcommand{\Lotimes}{\tens{\mathbf{L}}}
\newcommand{\Lotimesl}{\otimes^{\mathbf{L}}} 
\renewcommand\mod{\operatorname{mod}}

\newcommand\Perf{\operatorname{Perf}}
\newcommand{\RHom}{\mathbf{R}\!\operatorname{Hom}} 

\newcommand{\Dbf}{\mathbf{D}}
\newcommand{\Dbfb}{\mathbf{D}^{\textrm{b}}}
\newcommand{\Kbf}{\mathbf{K}}

\newcommand{\Hrm}{H^{\raisebox{0.2ex}{\scalebox{0.6}{$\bullet$}}}}

\newcommand{\Acal}{\mathscr{A}}
\newcommand{\Bcal}{\mathscr{B}}

\newcommand{\Pcal}{\mathscr{P}}

\newcommand{\Tcal}{\mathscr{T}}

\newcommand{\Xcal}{\mathscr{X}}

\newcommand{\dX}{{}^\dagger\hspace{-1.8pt}X}
\newcommand{\dXcal}{{}^\dagger\hspace{-1.8pt}\Xcal} 
\newcommand{\dP}{{}^\dagger\hspace{-1.8pt}P}
\newcommand{\dPcal}{{}^\dagger\hspace{-1.8pt}\Pcal}

\newcommand\arxiv[2]      {\href{http://arXiv.org/abs/#1}{#2}}
\newcommand\doi[2]        {\href{http://dx.doi.org/#1}{#2}}

\allowdisplaybreaks

\deffootnote[1em]{1em}{1em}
{\textsuperscript{\thefootnotemark}}

\theoremstyle{definition}
\newtheorem{definition}{Definition}

\newtheorem{defprop}[definition]{Definition and Proposition}
\newtheorem{proposition}[definition]{Proposition}
\newtheorem{theorem}[definition]{Theorem}
\newtheorem{lemma}[definition]{Lemma}
\newtheorem{corollary}[definition]{Corollary}
\newtheorem{remark}[definition]{Remark}

\newtheorem{assumption}[definition]{Assumption}

\numberwithin{equation}{section}
\numberwithin{definition}{section}
\numberwithin{figure}{section}

\newcommand\void[1]{}

\begin{document}

\title{Calabi-Yau completions \\ and orbifold equivalences}
\author{Nils Carqueville$^*$ \quad Alexander Quintero V\'{e}lez$^\dagger$
\\[0.5cm]
  \normalsize{\texttt{\href{mailto:nils.carqueville@univie.ac.at}{nils.carqueville@univie.ac.at}}} \\
  \normalsize{\texttt{\href{mailto:alexander.quintero@correounivalle.edu.co}{alexander.quintero@correounivalle.edu.co}}}\\[0.1cm]
  {\normalsize\slshape $^*$Erwin Schr\"odinger Institute \& Fakult\"at f\"ur Mathematik, Universit\"at Wien, Austria}\\[-0.1cm]
    {\normalsize\slshape $^\dagger$Departamento de Matem\'{a}ticas, Universidad del Valle, Colombia}\\[-0.1cm]
}
\date{}
\maketitle

\begin{abstract}
Calabi-Yau algebras are particularly symmetric differential graded algebras. 
There is a construction called `Calabi-Yau completion' \cite{Keller11} which produces a canonical Calabi-Yau algebra from any homologically smooth dg algebra. 
Homologically smooth dg algebras also form a 2-category to which the construction of `equivariant completion' of \cite{cr1210.6363} can be applied. In this theory two objects are called `orbifold equivalent' if there is a 
	1-morphism~$X$ 
with invertible quantum dimensions between them. Any such relation entails a whole family of equivalences between categories. 
%
	We show that an orbifold equivalence between two homologically smooth and proper dg algebras lifts to an orbifold equivalence between their Calabi-Yau completions under certain conditions on~$X$. 


\end{abstract}

\thispagestyle{empty}
\newpage

\tableofcontents

\section{Introduction and summary}\label{sec:intro}

In this paper we study the interaction between the notions of `Calabi-Yau completion' and `orbifold equivalence'. 
Let us begin by recalling these notions in round terms (Sections~\ref{sec:dgalgebras} and~\ref{sec:ECandOE} contain the precise definitions and further discussions). 

A differential graded (dg) algebra~$A$ over a field~$k$ is (homologically) smooth if it is perfect as a bimodule over itself; in other words, it is a compact object in the derived category $\Dbf(A^{\op} \otimes_k A)$. 
Being Calabi-Yau means being `very symmetric'. 
More precisely, for any integer~$n$ an $n$-Calabi-Yau algebra is a smooth dg algebra~$A$ with an isomorphism 
$$
A \cong \RHom_{A^{\op}\otimes_k A}(A,A^{\op}\otimes_k A) [n]
$$
in the derived category. So up to a shift~$A$ is its own dual as a bimodule. The theory of Calabi-Yau algebras is designed to capture and generalise properties of Calabi-Yau geometries at the level of derived categories. As such, they play prominent roles in mirror symmetry, noncommutative geometry, cluster algebras, and mathematical physics. 

Not every smooth dg algebra~$A$ is Calabi-Yau. However, there is a construction due to Keller \cite{Keller11} which produces a canonical $n$-Calabi-Yau algebra $\Pi_n(A)$: writing $\theta_A$ for the shift by $n-1$ of $\RHom_{A^{\op}\otimes_k A}(A,A^{\op}\otimes_k A)$, the $n$-Calabi-Yau completion $\Pi_n(A)$ is the tensor dg algebra 
\be\label{eq:CYcompletionintro}
\Pi_n(A) = A \oplus \theta_A \oplus \big(\theta_A \otimes_A \theta_A \big) \oplus\cdots \, . 
\ee
It really is a completion in the sense that $\Pi_n(A) \cong \Pi_n( \Pi_n(A) )$. 

Calabi-Yau algebras were extensively studied by Ginzburg in \cite{Ginzburg06}. In particular, there is a special type of Calabi-Yau algebra which can be built from any quiver with superpotential. These so-called Ginzburg algebras are surprisingly `dense' in $n$-Calabi-Yau algebras \cite{vdb1008.0599}, and in fact they exhaust all Calabi-Yau algebras in the important case of $n=3$.\footnote{The precise version of these statements involves \textsl{deformed} Calabi-Yau completions $\Pi_n(A,c)$ of \cite{Keller11}, where the differential on the tensor algebra~\eqref{eq:CYcompletionintro} is deformed by an element~$c$ of Hochschild homology. In the present paper we only consider the case $c=0$.}

Below we will exclusively consider acyclic quivers~$Q$ (which have zero superpotential) and their path algebras $A = kQ$, viewed as dg algebras with zero differential. The associated Ginzburg algebras were studied in \cite{Hermes14}. The simplest examples are Dynkin quivers $Q^{(\Gamma)}$ of ADE type~$\Gamma$, in which case the path algebras $\C Q^{(\Gamma)}$ are precisely the hereditary $\C$-algebras of finite representation type, and the corresponding Ginzburg algebras are certain twisted versions of preprojective algebras. 

\medskip

The general setting of the notion of orbifold equivalence is that of weak 2-categories (also called bicategories) with adjoints, 
while its origin lies in the study of symmetries and orbifolds for two-dimensional topological quantum field theories with defects (TQFT). 
As explained in \cite{dkr1107.0495}, any such TQFT~$Z$ (which by definition is a symmetric monoidal functor from a certain decorated bordism category to vector spaces) gives rise to a 2-category with adjoints $\B_Z$ that captures 
	much 
of the essence of~$Z$. Objects of $\B_Z$ are (interpreted as) closed TQFTs, 1-morphisms correspond to defect line operators interpolating between closed TQFTs, and 2-morphisms are point operators located at junctions where defect lines meet. 

In the pioneering work of Fr\"ohlich, Fuchs, Runkel, and Schweigert \cite{ffrs0909.5013} on rational conformal field theory, it was elucidated how the action of a finite symmetry group is encoded in a particular kind of Frobenius algebra. Transporting these ideas to TQFT lead in \cite{cr1210.6363} to a construction which produces from any bicategory with adjoints~$\B$ another such 2-category $\Beq$ which is called the equivariant completion of~$\B$. 
It really is a completion in the sense that $\Beq \cong (\Beq)_{\textrm{eq}}$. 

Objects of $\Beq$ are pairs $(a,A)$ with $a\in \B$ and $A: a \rightarrow a$ a separable Frobenius algebra in~$\B$, while the categories of 1-morphisms in $\Beq$ are given by bimodules and their intertwiners. In the special case when $\B = \B_Z$ for some TQFT~$Z$ and the Frobenius algebra~$A$ encodes the action of an orbifoldable symmetry group~$G$ on the closed subsector~$a$ of~$Z$, then $(a,A)\in\Beq$ describes the $G$-orbifold of~$a$. 
Not all separable Frobenius algebras `come from' symmetry groups (examples include those in~\eqref{eq:OEbetweenpathalgebras} below), and in this sense equivariant completion is a generalisation of symmetry.

The main result of \cite{cr1210.6363} is that if a 1-morphism $X: a \to b$ in~$\B$ has invertible `quantum dimension', then $A = X^\dagger \otimes X: a \to a$ is a separable Frobenius algebra, and there is an isomorphism 
\be\label{eq:aAcongbI}
(a,A) \cong (b,I_b) \quad \text{in } \Beq
\ee
where $I_b$ is the unit on~$b$. This relation between~$a$ and~$b$ induced by~$X$ is in fact an equivalence relation called orbifold equivalence, which we denote $a\sim b$. 

The isomorphism~\eqref{eq:aAcongbI} entails various interesting equivalences of categories. In particular, for every $c\in \B$ we have
\be\label{eq:OEconsequence}
\B(c,b) 
=
\Beq \big( (c,I_c), (b,I_b) \big) 
\cong 
\Beq \big( (c,I_c), (a, A) \big) \, . 
\ee
Since 1-morphisms in $\Beq$ are bimodules and everything is a bimodule over the unit $I_c$, this says that the category $\B(c,b)$ is equivalent to the category of those 1-morphisms $c \rightarrow a$ which have the structure of a right $A$-module. We stress that~$a$ and~$b$ may be very different from each other. 

To present examples of orbifold equivalences, let $\DGs_k$ denote the bicategory whose objects are smooth dg algebras and whose 1-morphism categories are the perfect derived categories, $\DGs_k(A,B) = \Perf( A^{\textrm{op}} \otimes_k B)$; further let $\DGsp_k$ be the subbicategory whose objects have finite-dimensional cohomology. In particular, the path algebras of Dynkin quivers $Q^{(\Gamma)}$ of ADE type~$\Gamma$ mentioned above are objects in $\DGsp_\C$. Employing the relation between Dynkin quivers and simple singularities \cite{kst0511155} together with the fact that quantum dimensions of matrix factorisations are easily computable thanks to the results of \cite{cm1208.1481}, it was shown in \cite{cr1210.6363, CRCR} that $Q^{(\Gamma)} \sim Q^{(\Gamma')}$ if and only if~$\Gamma$ and~$\Gamma'$ have the same Coxeter number. Put differently, in $\DGsp_\C$ we have orbifold equivalences
\begin{align}
& \C Q^{(\mathrm{A}_{2d-1})} \sim \C Q^{(\mathrm{D}_{d+1})}
\, , \quad \nonumber
\\
&\C Q^{(\mathrm{A}_{11})} \sim \C Q^{(\mathrm{E}_6)} 
\, , \quad
\C Q^{(\mathrm{A}_{17})} \sim \C Q^{(\mathrm{E}_7)} 
\, , \quad
\C Q^{(\mathrm{A}_{29})} \sim \C Q^{(\mathrm{E}_8)} 
\label{eq:OEbetweenpathalgebras}
\end{align}
as well as all the equivalences of type~\eqref{eq:OEconsequence}. 

\medskip

A natural question is whether the orbifold equivalences~\eqref{eq:OEbetweenpathalgebras} between path algebras lift to orbifold equivalences of their Calabi-Yau completions, i.\,e.~to their Ginzburg algebras. 
	In the present paper we give conditions under which the answer is affirmative (cf.~Corollary~\ref{cor:GinzburgDynkinOE}). 
	More generally, we describe conditions on when an orbifold equivalence $A \sim B$ in $\DGsp_k$ lifts to an orbifold equivalence $\Pi_n(A) \sim \Pi_n(B)$ in $\DGs_k$ between the $n$-Calabi-Yau completions (cf.~Theorem~\ref{thm:maintheorem}). 

There are two key ingredients to the proof. On the one hand, consistently using the natural 2-categorical language guides the argument and provides structural clarity. On the other hand, for technical computations we rely on K-projective resolutions to explicitly represent the bimodules we work with. 
In particular, for every 1-morphism $X: A \rightarrow B$ in $\DGsp_k$ we identify 
	canonical corresponding 1-morphisms $\Xcal, \Xcal': \Pi_n(A) \rightarrow \Pi_n(B)$ 
between the Calabi-Yau completions, and we express the quantum dimension of~$\Xcal$ in terms of those of~$X$
	-- under the assumption that $\Xcal \cong \Xcal'$ in $\DGs_k$. 
It is then immediate that if~$X$ has invertible quantum dimensions, i.\,e.~if it exhibits an orbifold equivalence $A \sim B$, then~$\Xcal$ also has invertible quantum dimension and provides an orbifold equivalence $\Pi_n(A) \sim \Pi_n(B)$. 
	%
	In particular, Ginzburg algebras for ADE type Dynkin quivers with the same Coxeter number are orbifold equivalent if the condition $\Xcal \cong \Xcal'$ is met. 

\medskip

We end this introduction with some further
	comments. 

\textsl{Twisted TQFTs. } 
Firstly, working with a 2-category of dg algebras $\DGsp_k$ is natural also from the point of view of TQFT: as shown in \cite{bfk1105.3177v3}, $\DGsp_k$ is equivalent to the 2-category $\DGsat_k$ of saturated dg categories. These include in particular Fukaya categories, derived categories of coherent sheaves, and categories of matrix factorisations -- or, put differently, topologically A- and B-twisted sigma models and Landau-Ginzburg models. More generally, we think of $\DGsat_k$ as the 2-category of ``all TQFTs arising from topologically twisting $\mathcal N=(2,2)$ supersymmetric quantum field theories'', with the differential in a dg category playing the role of the BRST operator. To understand mirror symmetry and other global properties of TQFTs, one should not study (the 2-categories of) sigma models or Landau-Ginzburg models individually, but rather all at once in a holistic conceptual framework. This is precisely what $\DGsat_k$ provides. Hence a general result on orbifold equivalences (which are a form of symmetry) and Calabi-Yau completions is of interest in this setting. 

\textsl{Exceptional unimodular singularities. } 
	We expect that there are many further interesting orbifold equivalences. 
In particular, there are four pairs among Arnold's 14 exceptional unimodular singularities whose Dynkin diagrams have the same Coxeter number, to wit $(E_{13},Z_{11})$, $(E_{14},Q_{10})$, $(Z_{13},Q_{11})$, and $(W_{13},S_{11})$ in the notation of \cite[Tab.\,5.2]{Ebelingbook}. It is natural to conjecture that these singularities are orbifold equivalent in the bicategory of Landau-Ginzburg 
	models.\footnote{This conjecture was proven in \cite{OEReck}, where also further orbifold equivalences were found.} This would have 
direct consequences for the $\mathcal N=2$ superconformal four-dimensional gauge theories studied in \cite{cdz1107.5747}. 

\textsl{Fukaya categories. }
A huge class of four-dimensional quantum field theories are related to the infamous six-dimensional $(0,2)$ superconformal field theories by compactifying the latter on certain punctured Riemann surfaces with meromorphic differential $(C,\phi)$ called `Gaiotto curves' \cite{Gaiotto07}. In this setting a quiver~$Q$ arises from triangulating (a blowup of) the Gaiotto curve, while on the other hand one can construct a symplectic 6-fold $Y_\phi$ as a conic fibration from the data $(C,\phi)$. It was shown in \cite{Smith13} that the bounded derived category of the Ginzburg algebra of~$Q$ for $n=3$ fully embeds into a certain Fukaya category of $Y_\phi$. 
It is natural to conjecture that 
	there are orbifold equivalences 
of symplectic 6-folds $Y_\phi$ of corresponding ADE types. 
In Remark~\ref{rem:Fukaya} we shall be more specific about this conjecture. 

\medskip

The remainder of the present paper is organised as follows. 
Section~\ref{sec:dgalgebras} and Appendix~\ref{app:dgalgebras} contain background material on dg algebras, their derived categories, Calabi-Yau completions and Ginzburg algebras. 
Section~\ref{sec:ECandOE} provides the basics on 2-categories, reviews the notions of equivariant completion and orbifold equivalence, as well as the above-mentioned examples involving simple singularities and Dynkin quivers; Propositions~\ref{prop:Xaction} and~\ref{prop:Aaction} also contain new details on these orbifold equivalences. 
In Section~\ref{sec:liftingOE} we present our main results, computing quantum dimensions in terms of Casimir elements and lifting orbifold equivalences in $\DGsp_k$ to Calabi-Yau completions.

\subsubsection*{Acknowledgements} 
We thank 
	Daniel Murfet, 
	Daniel Plencner, 
	Ingo Runkel, 
	and 
	Kazushi Ueda
for reading an earlier version of this manuscript, 
	and we are grateful to 
		Ivan Smith
	for explanations regarding Fukaya categories associated to meromorphic differentials. The work of N.\,C.~is partially supported by a grant from the Simons Foundation, and from the Austrian Science Fund (FWF): P\,27513-N27.

\section{Background on differential graded algebras}\label{sec:dgalgebras}

In this section we collect some of the standard definitions and results from the literature on differential graded algebras (or simply dg algebras) that are prerequisite for our studies in later sections. 
Section~\ref{subsec:perfectdual} discusses derived and perfect categories and their dualities, while in Section~\ref{subsec:CYcompletions} we recall the general notion of Calabi-Yau completion for dg algebras and the special case of Ginzburg algebras. 
Throughout~$k$ denotes a fixed field.  

\subsection{Perfect categories and duality}\label{subsec:perfectdual}

Good introductions to dg algebras and their modules include \cite{BL94,GM03,Pauksztello08}. For the reader's convenience as well as to fix our notation and conventions, we also provide a short review in Appendix~\ref{app:dgalgebras}.

\subsubsection{The derived category of a dg algebra}
\label{subsect:2.1}
We begin with a reminder on the derived category of a given dg algebra $A$ (over~$k$). 
A map of dg $A$-modules $f \colon M \to N$ is called \textsl{null-homotopic} if there exists a homomorphism of graded modules $h \colon M \to N$ of degree $+1$ such that $f= \nabla_N \circ h + h \circ \nabla_M$. The \textsl{homotopy category} $\Kbf (A)$ is defined as the category which has all dg $A$-modules as objects, and whose morphisms are equivalence classes of maps of dg $A$-modules up to homotopy. It carries the structure of a triangulated category and as such it comes equipped with a shift functor $[1]$ defined by $(M[1])^{i}=M^{i+1}$ and $\nabla_{M[1]}=-\nabla_M$ for every $i \in \ZZ$ and every dg $A$-module~$M$. A morphism of dg $A$-modules $s \colon M \to N$ is called a \textsl{quasi-isomorphism} if the induced morphism $\Hrm(s) \colon \Hrm(M) \to \Hrm(N)$ is an isomorphism of graded $k$-vector spaces. 

The \textsl{derived category} $\Dbf(A)$ is by definition the localisation of $ \Kbf(A)$ with respect to the class $\Sigma$ of all quasi-isomorphisms. This means that the derived category has the same objects as the homotopy category $\Kbf(A)$, and that morphisms in $\Dbf(A)$ are given by left fractions $s^{-1}\circ f$ with $s \in \Sigma$. We also note that the category $\Dbf(A)$ has a triangulated structure that is induced by the triangulated structure of $\Kbf(A)$. 

The derived category $\Dbf(A)$ contains a full subcategory formed by those dg $A$-modules $M$ whose total cohomology $\Hrm(M)$ is finite-dimensional as a graded $k$-vector space. We will denote this subcategory by $\Dbfb(A)$. In the case when $A$ is a finite-dimensional $k$-algebra, $\Dbfb(A)$ is identified with the bounded derived category of finitely generated left $A$-modules $\Dbfb(\mod (A))$.

\medskip
 
Next we shall discuss the lifting of the bifunctors $\Hom_A(-,-)$ and $-\otimes_A-$ to the derived category $\Dbf(A)$, for which we need to introduce some additional terminology. A dg $A$-module $P$ is said to be \textsl{K-projective} if the functor $\Hom_A(P,-)$ preserves quasi-isomorphisms. Dually, a dg $A$-module $I$ is said to be \textsl{K-injective} if the functor $\Hom_A(-,I)$ preserves quasi-isomorphisms. A \textsl{K-projective resolution} of a dg $A$-module $M$ is a K-projective dg $A$-module $P$ together with a quasi-isomorphism $\pi \colon P \to M$. Correspondingly, a \textsl{K-injective resolution} of $M$ is a K-injective dg $A$-module $I$ together with a quasi-isomorphism $\iota \colon M \to I$. Similar definitions hold for dg $A^{\op}$-modules. 

It can be shown that the derived category $\Dbf(A)$ has enough K-projective and enough K-injectives. This means that any object $M$ of $\Dbf(A)$ admits a K-projective resolution and a K-injective resolution. As a consequence, the bifunctors $\Hom_A(-,-)$ and $-\otimes_A-$ induce, respectively, a right derived bifunctor $\RHom_A(-,-)\colon \Dbf(A)^{\op} \times \Dbf(A) \to \Dbf(k)$ and a left derived bifunctor $-\Lotimesl_A-\colon \Dbf(A^{\op}) \times \Dbf(A) \to \Dbf(k)$, which can be computed as follows. Let~$M$ and~$N$ be dg $A$-modules, and consider the functors $\RHom_A(M,-)\colon \Dbf(A) \to \Dbf(k)$ and $\RHom_A(-,N)\colon \Dbf(A)^{\op}\to \Dbf(k)$. Take a K-projective resolution $\pi \colon P \to M$ and a K-injective resolution $\iota\colon N \to I$ in $\Dbf(A)$. We then have the following canonical isomorphisms:
$$
\RHom_A(M,-) \cong \Hom_A(P,-) 
\, ,\quad 
\RHom_A(-,N) \cong \Hom_A(-,I)
\, . 
$$
In similar fashion, let $M$ be a dg $A^{\op}$-module and $N$ a dg $A$-module, and consider the functors $M \Lotimesl_A-\colon \Dbf(A) \to \Dbf(k)$ and $-\Lotimesl_A N \colon \Dbf(A^{\op})\to \Dbf(k)$. Take a K-projective resolution $\pi \colon P \to M$ in $\Dbf(A^{\op})$ and a K-projective resolution $\rho \colon Q \to N$ in $\Dbf(A)$. Then we obtain the following canonical isomorphisms:
$$
M \Lotimes_A - \cong P \otimes_A -
\, , \quad 
-\Lotimes_A N \cong -\otimes_A Q
\, .
$$

As an aside, it should be noted that if $X$ is a dg $A^{\op} \otimes_k B$-module then we have induced right derived functors $\RHom_{A^{\op}}(X,-)\colon \Dbf(A^{\op}) \to \Dbf(B^{\op})$, $\RHom_{A^{\op}}(-,X)\colon \Dbf(A^{\op}) \to \Dbf(B)$, $\RHom_{B}(X,-) \colon  \Dbf(B) \to \Dbf(A)$ and $\RHom_{B}(-,X) \colon \Dbf(B)  \to \Dbf(A^{\op})$. We also have the left derived tensor functors $X \Lotimesl_A - \colon \Dbf(A) \to \Dbf(B)$ and $- \Lotimesl_B X \colon \Dbf(B^{\op}) \to \Dbf(A^{\op})$. 

We remark that the adjoint associativity law relating $\Hom$ and $\otimes$ can be upgraded to the derived category level. To be precise, take $A$ and $B$ to be dg algebras and $X$ to be a dg $A^{\op} \otimes_k B$-module. Then the left derived tensor functors $X \Lotimesl_A- \colon \Dbf(A) \to \Dbf(B)$ and $- \Lotimesl_B X \colon \Dbf(B^{\op}) \to \Dbf(A^{\op})$ are, respectively, the left adjoints of the right derived functors $\RHom_B(X,-)\colon \Dbf(B) \to \Dbf(A)$ and $\RHom_{A^{\op}}(X,-)\colon \Dbf(A^{\op}) \to \Dbf(B^{\op})$.

\subsubsection{The perfect category of a dg algebra}
\label{subsect:1.3}

We continue with a brief discussion of the perfect category of a dg algebra $A$. The reader is referred to \cite{Keller94,Keller06,Petit13} for a more complete exposition.

A dg $A$-module $M$ is said to be \textsl{perfect} if it can be obtained from $A$ using finitely many distinguished triangles, shifts, direct summands and finite coproducts. It can be shown that a dg $A$-module $M$ is perfect if and only if it is a compact object of $\Dbf(A)$, 
i.\,e.~if the functor $\Hom_{\Dbf(A)}(M,-)$ commutes with arbitrary coproducts. By $\Perf(A)$ we denote the full subcategory of $\Dbf(A)$ consisting of dg $A$-modules which are perfect. We will refer to it as the \textsl{perfect category} of the dg algebra $A$. If~$A$ is an ordinary $k$-algebra, then $\Perf(A)$ may be identified with the bounded homotopy category of finitely generated projective left $A$-modules. 

For future reference let us record the following simple but highly useful fact (cf.~\cite[Prop.\,3.10]{Petit13}). 

\begin{proposition}
\label{prop:2.2}
Let $A$ and $B$ be dg algebras and let $F \colon \Dbf(A) \to \Dbf(B)$ be a triangulated functor such that $F(A)$ is in $\Perf(B)$. Then, for any $M \in \Perf(A)$, its image $F(M)$ is in $\Perf(B)$.
\end{proposition}

We proceed to develop our vocabulary. A dg algebra $A$ is \textsl{proper} if it is a perfect dg $k$-module; this is equivalent to say that the cohomology $\Hrm(A)$ is finite-dimensional as a graded $k$-vector space. A dg algebra $A$ is \textsl{homologically smooth} if it is a perfect dg $A^{\op} \otimes_k A$-module. It is obvious that if $A$ is homologically smooth, then so is $A^{\op}$. We also note that if $A$ is an ordinary $k$-algebra then $A$ is homologically smooth if and only if $A$ has finite projective dimension as an $A$-$A$-bimodule. 

Suppose now that $A$ and $B$ are dg algebras and $X$ is a perfect dg $A^{\op} \otimes_k B$-module. If $A$ is proper, then it is not hard to see that $X$ is perfect as a dg $B$-module. Hence, applying Proposition \ref{prop:2.2}, we deduce that the left derived tensor functor $X \Lotimesl_A- \colon \Dbf(A) \to \Dbf(B)$ induces a functor $X \Lotimesl_A- \colon \Perf(A) \to \Perf(B)$. Similarly, if $B$ is proper, we have that $X$ is perfect as a dg $A^{\op}$-module and therefore the left derived tensor functor $-\Lotimesl_B X \colon \Dbf(B^{\op}) \to \Dbf(A^{\op})$ induces a functor $-\Lotimesl_B X \colon \Perf(B^{\op}) \to \Perf(A^{\op})$. In particular, if $A$ and $B$ are homologically smooth, then so is their tensor product $A \otimes_k B$.

We also need the following result by To\"{e}n and Vaqui\'{e} \cite[Lem.\,2.8.2]{TV07} which allows us to prove perfectness of a dg bimodule.

\begin{proposition}
\label{prop:2.3}
Let $A$ and $B$ be dg algebras and let $X$ be a dg $A^{\op} \otimes_k B$-module. If $A$ is homologically smooth and the left derived tensor functor $X \Lotimesl_A- \colon \Dbf(A) \to \Dbf(B)$ preserves perfect dg modules, then $X$ is a perfect dg $A^{\op} \otimes_k B$-module.
\end{proposition} 

The following fact (see e.\,g.~\cite[Thm.\,4.1.6]{Pauksztello08}) will also be useful to us in Section~\ref{sec:liftingOE}.  

\begin{proposition}
\label{prop:2.4}
Let $A$ and $B$ be dg algebras and let $X$ be a perfect dg $A^{\op} \otimes_k B$-module. 
\begin{enumerate}
\item For each dg $A^{\op}$-module $M$, there is a natural isomorphism
$$
M \Lotimes_A \RHom_{A^{\op}}(X, A) \xlongrightarrow{\cong} \RHom_{A^{\op}}(X, M) 
$$
which is functorial in $M$.
\item For each dg $B$-module $N$, there is a natural isomorphism 
$$
\RHom_B(X, B) \Lotimes_B N \xlongrightarrow{\cong} \RHom_B (X, N) 
$$
which is functorial in $N$.
\end{enumerate}
\end{proposition}

It will prove convenient for our purposes to isolate representatives of these canonical maps in the homotopy category in the case in which $A$ and $B$ are both assumed to be proper. In this regard, it is useful to observe that a K-projective object of $\Dbf(A^{\op} \otimes_k B)$ is also a K-projective object of $\Dbf(A^{\op})$ and $\Dbf(B)$, forgetting the dg $B$-module and dg $A^{\op}$-module structure, respectively. Hence, if we take a K-projective resolution $\pi\colon P \to X$ of $X$ in $\Dbf(A^{\op} \otimes_k B)$, we will be able to conclude that it is a K-projective resolution of $X$ both as a dg $A^{\op}$-module and as a dg $B$-module. With this understood, the canonical morphism of Proposition~\ref{prop:2.4}(i) will be represented by the map in $\Kbf(B)$ given by 
\begin{align*}
\zeta \colon M \otimes_A \Hom_{A^{\op}}(P,A) & \lra \Hom_{A^{\op}}(P,M) \, , 
\\
m \otimes f & \longmapsto \big( p \mapsto m f(p) \big)
\, . 
\end{align*}
In the same way, the canonical morphism of Proposition \ref{prop:2.4}(ii) is represented by the map in $\Kbf(A^{\op})$ given by 
\begin{align*}
\widetilde{\zeta} \colon \Hom_B(P,B) \otimes_B N & \lra \Hom_B(P,N) \, , 
\\
g \otimes n & \longmapsto \big( p \mapsto (-1)^{\vert p \vert \vert n \vert} g(p) n \big) \, .
\end{align*}

\subsubsection{Duality for perfect dg modules}
\label{subsect:2.3}

We conclude this subsection with a short examination of duality for perfect dg modules. For details, the reader is referred to \cite{Keller08,Shklyarov07,Shklyarov13,Petit13}.

Let $A$ be a homologically smooth dg algebra. We define the \textsl{inverse dualising complex} $\Theta_A$ of $A$ to be any K-projective resolution of $\RHom_{A^{\op}\otimes_k A}(A,A^{\op}\otimes_k A)$ considered as an object of $\Dbf(A^{\op}\otimes_k A)$. In view of the discussion in Section~\ref{subsect:2.1}, we have a canonical isomorphism
$$
\RHom_{A^{\op}\otimes_k A}(A,A^{\op}\otimes_k A) \Lotimes_A - \cong \Theta_A \otimes_A - 
$$
and thus we get a triangulated functor $\Theta_A \otimes_A - \colon \Dbf(A) \to \Dbf(A)$.  It can be shown, using a variation of \cite[Lem.\,4.1]{Keller08}, that this functor is a quasi-inverse of a Serre functor. This means that for any object~$M$ of $\Dbfb(A)$ and any object~$N$ of $\Dbf(A)$, there is a canonical isomorphism
$$
\Hom_{\Dbf(A)}(\Theta_A \otimes_A N,M) \xlongrightarrow{\cong} \Hom_{\Dbf(A)}(M,N)^*
$$
where $(-)^*$ denotes the dual with respect to $k$.  

If in addition to being homologically smooth, $A$ is also proper, one can give an explicit description of the Serre functor on $\Dbf(A)$. To this end, put $(A^{\op})^*=\Hom_k(A^{\op},k)$, where $k$ is thought of as a dg $k$-module concentrated in degree~$0$. We define the \textsl{dualising complex} $\Omega_A$ of $A$ to be any K-projective resolution of $(A^{\op})^*$ considered as an object of $\Dbf(A^{\op} \otimes_k A)$. Since it is easy to see that $(A^{\op})^*$ is a perfect dg $A^{\op} \otimes_k A$-module, we get a well-defined triangulated functor $\Omega_A \otimes_A - \colon \Perf(A) \to \Perf(A)$. It is shown in \cite[Thm.\,3.28]{Petit13} (see also \cite[Thm.\,4.3]{Shklyarov07}) that this functor is a Serre functor. On the other hand, it is also easily checked that  $\RHom_{A^{\op}\otimes_k A}(A,A^{\op}\otimes_k A)$ is a perfect dg $A^{\op}\otimes_k A$-module, and so we obtain a well-defined triangulated functor $\Theta_A \otimes_A- \colon \Perf(A) \to \Perf(A)$. Then, as shown in \cite[Thm.\,3.31]{Petit13} (see also \cite[Thm.\,4.4]{Shklyarov07}), the functors $\Omega_A \otimes_A -$ and $\Theta_A \otimes_A -$ from $\Perf(A)$ to $\Perf(A)$ are mutually quasi-inverse equivalences. From this we deduce the existence of canonical isomorphisms of perfect dg $A^{\op}\otimes_k A$-modules
\begin{equation}
\label{eqn:2.1}
\Omega_A \otimes_A \Theta_A \cong A 
\, , \quad 
\Theta_A \otimes_A \Omega_A \cong A
\, . 
\end{equation}

\medskip

Suppose now that $A$ and $B$ are homologically smooth and proper dg algebras and $X$ is a perfect dg $A^{\op}\otimes_k B$-module. Then, as observed in Section~\ref{subsect:1.3}, we get an induced functor $X \Lotimesl_A -\colon \Perf(A) \to \Perf(B)$. One can use the Serre duality results we have just described to prove that $X \Lotimesl_A -$ admits both a left and a right adjoint. To be more specific, set 
$$
X^{\vee}=\RHom_{A^{\op}\otimes_k B}(X,A^{\op}\otimes_k B)
$$ 
and define
\begin{equation}
\label{eqn:2.2}
{}^{\dagger}\! X=X^{\vee} \otimes_B \Omega_B
\, , \quad 
X^{\dagger}=\Omega_A \otimes_A X^{\vee}
\, . 
\end{equation}
It is not hard to verify that ${}^{\dagger}\! X$ and $X^{\dagger}$ are both perfect dg $B^{\op}\otimes_k A$-modules and therefore we get induced functors ${}^{\dagger}\! X \Lotimesl_B - \colon \Perf(B) \to \Perf(A)$ and $X^{\dagger} \Lotimesl_B - \colon \Perf(B) \to \Perf(A)$. The next result is taken from \cite[Lem.\,A.20]{bfk1105.3177v3}. 

\begin{proposition}
\label{prop:2.5}
Let the setting be as above. Then the functors 
$$
{}^{\dagger}\! X \Lotimes_B - \colon \Perf(B) \lra \Perf(A)
\, , \quad 
X^{\dagger} \Lotimes_B - \colon \Perf(B) \lra \Perf(A)
$$
are respectively the right and left adjoints to $X \Lotimesl_A -\colon \Perf(A) \to \Perf(B)$.
\end{proposition}

This result will have profound implications when we discuss bicategories of dg algebras in Section~\ref{sec:liftingOE}. Note that, in view of the relation \eqref{eqn:2.1}, there is an isomorphism of perfect  dg $B^{\op}\otimes_k A$-modules
\begin{equation}
\label{eqn:2.3}
{}^{\dagger}\! X \cong \Theta_A \otimes_A X^{\dagger} \otimes_B \Omega_B \, .
\end{equation}
This, of course, is a reflection of the fact that the functors ${}^{\dagger}\! X \Lotimesl_B -$ and $X^{\dagger} \Lotimesl_B -$ are related by conjugation with the Serre functors. 

\subsection{Calabi-Yau completions and Ginzburg algebras}
\label{subsec:CYcompletions}

We now briefly recall the definition and some relevant properties of the notion of Calabi-Yau completion (due to Keller \cite{Keller11}) of a homologically smooth dg algebra, as well as the Ginzburg algebra (due to Ginzburg \cite{Ginzburg06}) associated to a quiver. 
For further discussion and proofs of the results collected below we refer to these two references.

\subsubsection{Calabi-Yau completions}

Let $A$ be a homologically smooth dg algebra, let~$n$ be an integer, and consider the inverse dualising complex $\Theta_A$ of~$A$, i.\,e.~a K-projective resolution of $\RHom_{A^{\op}\otimes_k A}(A,A^{\op}\otimes_k A)$ as in Section~\ref{subsect:2.3}. We say that~$A$ is an $n$\textsl{-Calabi-Yau algebra} if, in the derived category $\Dbf(A^{\op}\otimes_k A)$, there is an isomorphism
\be\label{eq:CYkappa}
A 
\xlongrightarrow{\cong}
\Theta_A [n] \, .
\ee
This terminology is justified by the fact that this property implies that $\Dbfb(A)$ is $n$-Calabi-Yau as a triangulated category, see e.\,g.~\cite[Lem.\,4.1]{Keller08}.\footnote{A $k$-linear triangulated category~$\mathcal D$ is $n$-Calabi-Yau if it admits a Serre functor~$S$ and there is an isomorphism between~$S$ and the $n$-fold iteration of the shift functor $[1]$. Hence there are isomorphisms $\Hom_{\mathcal D}(X,Y) \cong \Hom_{\mathcal D}(Y,X[n])^*$ natural in both~$X$ and~$Y$. The bounded derived category of coherent sheaves on an smooth projective Calabi-Yau variety of dimension~$n$ is an example of a triangulated $n$-Calabi-Yau category.} 

\begin{definition}
Let~$A$ be a homologically smooth dg algebra, fix $n\in\Z$, and set $\theta_A = \Theta_A[n-1]$. The $n$\textsl{-Calabi-Yau completion} of $A$ is the tensor dg algebra
$$
\Pi_n(A)= T_A (\theta_A)= A \oplus \theta_A \oplus \big(\theta_A \otimes_A \theta_A \big) \oplus\cdots \, ,
$$
with differential acting componentwise. 
\end{definition}

One can check that up to quasi-isomorphism, $\Pi_n(A)$ is independent of the choice of K-projective resolution made in the definition of $\Theta_A$. 
Moreover, $\Pi_n(-)$ really is a completion in the sense that there is a quasi-isomorphism of dg algebras between $\Pi_n(A)$ and $\Pi_n( \Pi_n(A) )$, so in particular 
$$
\Pi_n(A) \cong \Pi_n \big( \Pi_n(A) \big) 
\quad \text{in } \DGs_k \, .
$$

In the case in which~$A$ is the path algebra of a non-Dynkin quiver and $n=2$, it can be seen that $\Pi_n(A)$ is quasi-isomorphic to the preprojective algebra of $A$.  For this reason, the $n$-Calabi-Yau completion $\Pi_n(A)$ is sometimes also referred to as the \textsl{derived $n$-preprojective algebra}. 

Since the canonical injection $A \to \Pi_n(A)$ is a map of dg algebras, we can endow $\Pi_n(A)$ with either a dg $A$-module or a dg $A^{\op}$-module structure. It is easily verified that $\Pi_n(A)$ is perfect and K-projective both as a dg $A$-module and as a dg $A^{\op}$-module. 

The main result of \cite{Keller11} justifies the name `Calabi-Yau completion': 

\begin{theorem}
Let~$A$ be a homologically smooth dg algebra and $n\in\Z$. Then the $n$-Calabi-Yau completion $\Pi_n(A)$ is homologically smooth and Calabi-Yau as a dg algebra. 
\end{theorem}

In particular, $\Dbfb(\Pi_n(A))$ is an $n$-Calabi-Yau triangulated category. 

\medskip

It is worth mentioning that the construction $A \mapsto \Pi_n(A)$ has a geometric counterpart. To wit, let $Z$ be a smooth algebraic variety of dimension $n-1$ and let $X= \mathrm{tot}(\omega_Z)$ be the total space of the canonical line bundle of~$Z$. It is a well known fact that~$X$ has trivial canonical bundle and hence is a smooth Calabi-Yau variety of dimension~$n$. 
One can think of~$X$ as the `Calabi-Yau completion' of~$Z$ also in the following sense: 
let~$A$ be the endomorphism algebra of a tilting object in the bounded derived category of coherent sheaves on~$Z$. Then the bounded derived category of coherent sheaves on~$X$ is equivalent to the derived category of the $n$-Calabi-Yau completion $\Pi_n(A)$ (see e.\,g.~\cite[Thm.\,3.6]{BS10}).

\subsubsection{Ginzburg algebras}

We now wish to recall the definition of the Ginzburg algebra associated to an acyclic quiver, for which again we require some terminology. 
A \textsl{quiver} $Q$ is an oriented graph, specified by a set of \textsl{vertices} $Q_0$, a set of \textsl{arrows} $Q_1$, and two maps $h,t \colon Q_1 \to Q_0$  which associate to each arrow $a \in Q_1$ its \textsl{head} $h(a) \in Q_0$ and \textsl{tail} $t(a) \in Q_0$. A \textsl{path} in $Q$ of length $n$ is an ordered sequence of arrows $p=a_n \cdots a_1$ such that $h(a_{\nu})=t(a_{\nu+1})$ for $1 \leqslant \nu \leqslant n-1$. For each vertex $i \in Q_0$, we let $e_i$ denote the trivial path with $h(e_i)=t(e_i)=i$. For any path $p=a_n \cdots a_1$ we set $h(p)=h(a_n)$ and $t(p)=t(a_1)$. If $h(p)=t(p)$, then we say that $p$ is an \textsl{oriented cycle}. An oriented cycle of length $1$ is also called a \textsl{loop}. We call the quiver $Q$ \textsl{acyclic} if $Q$ has no oriented cycles. The \textsl{path algebra} $k Q$ of $Q$ is the $k$-algebra whose underlying $k$-vector space has as its basis the set of all paths in $Q$ and the product of two basis vectors is defined by concatenation.  

\begin{definition}
Let $Q$ be an acyclic quiver and let $n \geqslant 2$ be an integer. We denote by $\widehat{Q}$ the quiver obtained from $Q$ by adding a reverse arrow $a^*\colon j \to i$ for each arrow $a \colon i \to j$ in $Q$ and an additional loop $t_i$ for each vertex $i \in Q_0$. The \textsl{Ginzburg algebra} $\Gamma_n(Q)$ is the dg algebra whose underlying graded algebra is the path algebra $k \widehat{Q}$ with degrees of the generators being $\vert a \vert=0$ and $\vert a^*\vert=n-2$ for all $a \in Q_1$, and $\vert t_i \vert=n-1$ for all $i \in Q_0$. The differential~$d$ on $\Gamma_n(Q)$ is uniquely determined by the fact that is $k$-linear, satisfies the Leibniz rule, and acts on arrows of $\widehat{Q}$ as follows:
$$
d a= da^* =0
\, , \quad 
d t_i= \sum_{a \in Q_1} e_i \big[ a, a^* \big] e_i 
$$ 
where $[a, a^*]$ is the commutator $a a^*-a^* a$.
\end{definition}

The above definition can be extended to any `quiver with superpotential', see \cite[Sect.\,4.3]{Ginzburg06} or \cite[Sect.\,6.2]{Keller11} for details. However, in the present paper we restrict ourselves to the case of acyclic quivers, where there are only zero superpotentials. 

The following important result obtained in \cite[Thm.\,6.3]{Keller11} describes the link between Calabi-Yau completions and Ginzburg algebras.

\begin{theorem}
\label{thm:2.8}
Let $Q$ be an acyclic quiver and let $n \geqslant 2$ be an integer. Then the $n$-Calabi-Yau completion $\Pi_n(k Q)$ of the path algebra $k Q$ is quasi-isomorphic to the Ginzburg algebra $\Gamma_n(Q)$. In particular, the Ginzburg algebra $\Gamma_n(Q)$ is homologically smooth and $n$-Calabi-Yau as a dg algebra. 
\end{theorem}

As shown in \cite{Hermes14}, if~$Q$ is an acyclic non-Dynkin quiver then the Ginzburg algebra $\Gamma_3(Q)$ is formal and quasi-isomorphic to the preprojective algebra of~$Q$. By contrast, in the case when~$Q$ is a Dynkin quiver, $\Gamma_3(Q)$ is quasi-isomorphic to a certain twist of a polynomial algebra over the preprojective algebra of~$Q$; it is not formal but admits an $A_{\infty}$-minimal model whose only non-zero products are $\mu_2$ and $\mu_3$.

\section{Equivariant completion and orbifold equivalence}\label{sec:ECandOE}

In this section we review the theory of equivariant completion introduced in \cite{cr1210.6363}. Section~\ref{subsec:ECtheory} also contains the basic bicategorical algebra needed, and in Section~\ref{subsec:appLG} we discuss the application to Landau-Ginzburg models of ADE type from \cite{cr1210.6363, CRCR} which we shall lift to Ginzburg algebras in Section~\ref{subsec:liftingOEDynkin}.

\subsection{General theory}\label{subsec:ECtheory}

\subsubsection{Bicategorical algebra}

Here we collect some basic definitions and fix our notation. For more on bicategories we refer to \cite{bor94}. 

A \textsl{bicategory}~$\B$ is a category weakly enriched over Cat. More precisely, it has a collection of \textsl{objects}~$\B$, and for every ordered pair of objects $a,b$ there is a category $\B(a,b)$ whose objects and maps are called \textsl{1-morphisms} and \textsl{2-morphisms}, respectively. These come with functors $\otimes: \B(b,c) \times \B(a,b) \rightarrow \B(a,c)$ which provide for \textsl{horizontal composition} of 1- and 2-morphisms. 
Horizontal composition is associative and unital in the sense that there are natural 2-isomorphisms $\alpha_{X,Y,Z}: (X \otimes Y) \otimes Z \rightarrow X \otimes (Y \otimes Z )$ for any triple of composable 1-morphisms $X,Y,Z$, and for every $a\in \B$ there is the \textsl{unit} 1-morphism $I_a \in \B(a,a)$ together with natural 2-isomorphisms $\lambda_X: I_b \otimes X \rightarrow X$ and $\rho_X: X \otimes I_a \rightarrow X$ for all $X \in \B(a,b)$. To complete the definition of~$\B$ these data have to satisfy two coherence axioms which are written out in \cite[(7.18)\,\&\,(7.19)]{bor94}. 

Two standard examples of bicategories are those whose objects, 1-, 2-morphisms are categories, functors, natural transformations and rings, bimodules, bimodule maps, respectively, with no surprises regarding compositions and units. In Section~\ref{subsec:appLG} we will meet the bicategory $\LGgr$ of (graded) Landau-Ginzburg models, and our main result concerns the bicategory $\DGs_k$ of homologically smooth dg algebras, see Definition~\ref{def:DGs}.  

A 1-morphism $X\in \B(a,b)$ in a bicategory~$\B$ has a \textsl{left adjoint} if there is $\dX\in \B(b,a)$ together with \textsl{adjunction maps}
$$
\ev_X: \dX  \otimes X \lra I_a
\, , \quad
\coev_X: I_b \lra X \otimes \dX 
$$
which satisfy the \textsl{Zorro moves}
\begin{align}
& \rho_X \circ (1_X \otimes \ev_X) \circ \alpha_{X, {}^\dagger\!X, X} \circ (\coev_X \otimes 1_X) \circ \lambda_X^{-1} = 1_X \, , \nonumber \\
& \lambda_{\dX} \circ( \ev_X \otimes 1_{\dX}) \circ \alpha_{\dX, X, \dX}^{-1} \circ (1_{\dX} \otimes \coev_X) \circ \rho_{\dX}^{-1} = 1_{\dX} \, . \label{eq:Zorro}
\end{align}
A \textsl{right adjoint} of~$X$ is given by $X^\dagger \in \B(b,a)$ and maps 
$$
\tev_X: X \otimes X^\dagger \lra I_b
\, , \quad
\tcoev_X: I_a \lra X^\dagger \otimes X 
$$
constrained by Zorro moves analogous to those in \eqref{eq:Zorro}. 

\begin{definition}
If every 1-morphism in~$\B$ has left and right adjonts, we say that~$\B$ is a \textsl{bicategory with adjoints}. A 1-morphism~$X$ is called \textsl{ambidextrous} if it has left and right adjoints such that there is a 2-isomorphism $\alpha_X: \dX \rightarrow X^\dagger$. 
\end{definition}

The following notion, which generalises the (finite) dimension of a vector space, will be central for our purposes: 

\begin{definition}\label{def:qdims}
For an ambidextrous 1-morphism $X\in \B(a,b)$ with isomorphism $\alpha_X: \dX \rightarrow X^\dagger$ its \textsl{left} and \textsl{right quantum dimensions} are respectively the 2-morphisms
\begin{align*}
& \diml(X) = \ev_X \circ ( \alpha_X^{-1} \otimes 1_X )\circ \tcoev_X \in \End(I_a) \, , 
\\
& \dimr(X) = \tev_X \circ ( 1_X \otimes \alpha_X )\circ \coev_X \in \End(I_b) \, .
\end{align*}
\end{definition}


\medskip

For any bicategory~$\B$, a 1-morphism $A \in \B(a,a)$ is an \textsl{algebra} if it comes with 2-morphisms $\mu: A\otimes A \rightarrow A$ and $\eta: I_a \rightarrow A$ which give a unital associative structure, i.\,e.
$$
\mu \circ (\mu \otimes 1_A) = \mu \circ (1_A \otimes \mu ) \, , 
\quad
\mu \circ (\eta \otimes 1_A) = 1_A = \mu \circ (1_A \otimes \eta) \, . 
$$
Similarly, $A$ is a \textsl{coalgebra} if it comes with 2-morphisms $\Delta: A \rightarrow A \otimes A$ and $\varepsilon: A \rightarrow I_a$ which are coassociative and counital. 

\begin{definition}
Let $A\in \B(a,a)$ have both an algebra and a coalgebra structure as above. 
\begin{enumerate}
\item
$A$ is \textsl{Frobenius} if 
$
( 1_A \otimes \mu ) \circ ( \Delta \otimes 1_A ) = \Delta \circ \mu = ( \mu \otimes 1_A ) \circ ( 1_A \otimes \Delta ) 
$. 
\item
$A$ is \textsl{$\Delta$-separable} (or simply \textsl{separable}) if 
$
\mu \circ \Delta = 1_A 
$. 
\end{enumerate}
\end{definition}

Given (any) algebra $A \in \B(a,a)$ with structure maps $\mu, \eta$, a \textsl{right $A$-module} is a 1-morphism $X \in \B(a,b)$ together with a 2-morphism $\rho: X \otimes A \rightarrow X$ called \textsl{right $A$-action} which is compatible with the algebra structure on~$A$, i.\,e.
$$
\rho \circ ( \rho \otimes 1_A ) = \rho \circ (1_X \otimes \mu ) \, , 
\quad 
\rho \circ (1_X \otimes \eta ) = 1_X \, . 
$$
A \textsl{module map} between right $A$-modules~$X$ and~$X'$ is a 2-morphism $X\rightarrow X'$ which commutes with the two right $A$-actions. 
We write $\modu(A)$ for the category of such right $A$-modules if the `parent category' $\B(a,b)$ is clear from the context. 

Similarly, a \textsl{left $A$-module} is $Y\in \B(c,a)$ together with a 2-morphism $\lambda: A \otimes Y \rightarrow Y$ called \textsl{left $A$-action} compatible with~$\mu$ and~$\eta$, and a module map between left $A$-modules~$Y$ and~$Y'$ is a 2-morphism $Y \rightarrow Y'$ which commutes with the two left $A$-actions. 

If $B\in \B(b,b)$ is another algebra then a \textsl{$B$-$A$-bimodule} is a 1-morphism $X\in \B(a,b)$ which has the structure of a right $A$-module and a left $B$-module such that the actions of~$A$ and~$B$ commute. A \textsl{bimodule map} between two $B$-$A$-bimodules is simultaneously a map of right $A$- and left $B$-modules. 

We note that if $A,B$ are both Frobenius algebras, then the category of $B$-$A$-bimodules has a Serre functor \cite[Prop.\,3.12]{BCP1}, given by twisting the left $B$-action by a map $\gamma_B$ and the right $A$-action by $\gamma_A^{-1}$, where $\gamma_A = ( \ev_A \otimes 1_A) \circ ( 1_{A^\dagger} \otimes [ \Delta \circ \eta \circ \varepsilon \circ \mu ] ) \circ ( \tcoev_A \otimes 1_A )$ is the \textsl{Nakayama automorphism}. 

\medskip

Let $A\in \B(a,a)$ be an algebra, $X\in \B(a,b)$ a right $A$-module and $Y \in \B(c,a)$ a left $A$-module with $A$-actions~$\rho$ and~$\lambda$, respectively. Then the \textsl{tensor product of~$X$ and~$Y$ over~$A$} is the coequaliser of $\rho \otimes 1_Y$ and $(1_X \otimes \lambda) \circ \alpha_{X,A,Y}$: 
$$
\begin{tikzpicture}[
			     baseline=(current bounding box.base), 
			     >=stealth,
			     descr/.style={fill=white,inner sep=2.5pt}, 
			     normal line/.style={->}
			     ] 
\matrix (m) [matrix of math nodes, row sep=3em, column sep=2.5em, text height=1.5ex, text depth=0.25ex] {%
(X \otimes A) \otimes Y & & X \otimes Y & & X \otimes_A Y \\
&& && Z \\ };
\path[font=\scriptsize] (m-1-1) edge[->, transform canvas={yshift=0.5ex}] node[auto] {$ \rho \otimes 1 $} (m-1-3)
				  (m-1-1) edge[->, transform canvas={yshift=-0.5ex}] node[auto, swap] {$ (1 \otimes \lambda) \circ \alpha $} (m-1-3); 
\path[font=\scriptsize] 
				 (m-1-3) edge[->] node[auto] {} (m-1-5)
				 (m-1-3) edge[->] node[auto, swap] {} (m-2-5);
\path[font=\scriptsize, dotted] (m-1-5) edge[->] node[auto] {$ \exists ! $} (m-2-5);
\end{tikzpicture}
$$

If idempotent 2-morphisms split in~$\B$ (which is the case for both $\B = \LGgr$ and $\B = \DGs_k$) and~$A$ is also separable Frobenius, then $X \otimes_A Y$ exists and may be computed as the image of the projector $( \rho \otimes \lambda) \circ ( 1_X \otimes [ \Delta \circ \eta ] \otimes 1_Y )$, see e.\,g. \cite[Sect.\,2.3]{cr1210.6363} for details.

\subsubsection{Equivariant completion}

Given a bicategory~$\B$ we will construct a new one into which~$\B$ fully embeds. The construction below first appeared in \cite{cr1210.6363} and is motivated by the study of (generalised) orbifolds in two-dimensional topological quantum field theory, following the pioneering work of \cite{ffrs0909.5013} on rational conformal field theory.\footnote{Very roughly, the conditions on the algebras appearing in our Definition~\ref{def:Beq} encode invariance under certain triangulations of two-dimensional decorated bordisms.} 

\begin{defprop}\label{def:Beq}
Let~$\B$ be a bicategory whose categories of 1-morphisms are idempotent complete. The \textsl{equivariant completion} $\Beq$ of~$\B$ is the bicategory which consists of the following data: 
\begin{itemize}
\item
Objects are pairs $(a,A)$ with $a\in \B$ and $A \in \B(a,a)$ a separable Frobenius algebra. 
\item 
1-morphisms $(a,A) \rightarrow (b,B)$ are $B$-$A$-bimodules in $\B(a,b)$. 
\item 
2-morphisms are bimodule maps. 
\item 
The horizontal composition of $Y: (c,C) \rightarrow (a,A)$ and $X: (a,A) \rightarrow (b,B)$ is the tensor product $X \otimes_A Y$, and the associator is induced from the one in~$\B$. 
\item 
The unit $I_{(a,A)}$ for $(a,A) \in \Beq$ is given by~$A$, with the left and right actions on 1-morphisms given by their left and right $A$-actions. 
\end{itemize}
\end{defprop}

The attribute ``equivariant'' has its origin in the relation to the orbifold constructions mentioned above. The name ``completion'' is appropriate since not only does one have the full embedding $\B \subset \Beq$ given by $a \mapsto (a,I_a)$, but there is also an essentially surjective full embedding (see \cite[Prop.\,4.2]{cr1210.6363})
$$
\Beq \cong (\Beq)_{\textrm{eq}} \, . 
$$
In other words, $(-)_{\textrm{eq}}$ is an idempotent operation. 

One way to construct separable Frobenius algebras in a given bicategory is from the action of a finite group~$G$ (in which case the categories of bimodules in Definition~\ref{def:Beq} are equivalent to $G$-representations), see e.\,g.~\cite[Sect.\,7.1]{cr1210.6363}. Another construction, which is the main result of \cite{cr1210.6363} and key for our present paper, involves ambidextrous 1-morphisms with a special invertibility property: 

\begin{theorem}
Let~$\B$ be a bicategory and $X\in \B(a,b)$ an ambidextrous 1-morphism such that $\diml(X)$ and $\dimr(X)$ are isomorphisms. 
\begin{enumerate}
\item
$A := X^\dagger \otimes X \in \B(a,a)$ is a separable Frobenius algebra. 
\item 
$X: (a,A) \rightleftarrows (b,I_b) : X^\dagger$ is an adjoint equivalence in $\Beq$. 
\end{enumerate}
\end{theorem}

Let us now assume that~$\B$ is a $k$-linear bicategory for some field~$k$, such that all quantum dimensions are multiples of the identity. (This is the case for the applications in Sections~\ref{subsec:appLG} and~\ref{subsec:liftingOEDynkin}; see \cite[Rem.\,2.3]{CRCR} for a more general discussion.) Then the above theorem describes an equivalence relation: 

\begin{defprop}
Two objects $a,b$ in a bicategory as above are \textsl{orbifold equivalent}, $a\sim b$, if there is a 1-morphism~$X$ with invertible quantum dimensions between them. We say that~$X$ \textsl{exhibits} the equivalence $a \sim b$. 
\end{defprop}

The merit of an orbifold equivalence $X:a\sim b$ in~$\B$ is that for every $c\in \B$ one has equivalences of categories 
\be\label{eq:con-equi}
\B(c,b) 
=
\Beq \big( (c,I_c), (b,I_b) \big) 
\cong 
\Beq \big( (c,I_c), (a, X^\dagger \otimes X) \big) \, . 
\ee
Put differently, everything in~$\B$ that relates to~$b$ can be expressed in terms of modules over the algebra $X^\dagger \otimes X$ on~$a$. This is particularly useful if~$b$ is `complicated' while~$a$ and~$A$ are `easy'.

\subsection{Applications to Landau-Ginzburg models}\label{subsec:appLG}

\subsubsection{Bicategory of affine Landau-Ginzburg models}

For every commutative ring~$k$ there is a bicategory of (affine) Landau-Ginzburg models $\LG$. Here we will briefly review the case $k=\C$ and refer to \cite{cm1208.1481} for details. 

Objects of the bicategory $\LG$ are polynomial rings $R = \C[x_1,\ldots,x_n]$ in any number of variables together with an isolated singularity $W\in R$, i.\,e.~$\dim_\C(R/(\partial_{x_1}W, \ldots, \partial_{x_n}W)) < \infty$. We sometimes simply write~$W$ for the object $(R,W)$. A \textsl{matrix factorisation} of $(R,W)$ is a finitely generated free $\Z_2$-graded $R$-module $X = X^0 \oplus X^1$ together with an odd $R$-linear endomorphism $d_X$ (called \textsl{twisted differential}) such that $d_X^2 = W \cdot 1_X$. If the $R$-module $X^0$ (and thus, for $W\neq 0$, also $X^1$) has rank~$r$, then $(X,d_X)$ is called \textsl{rank-$r$}. Matrix factorisations of $(R,W)$ together with even linear maps up to homotopy with respect to the twisted differentials form a triangulated category $\hmf(R,W)$; for its idempotent closure we write $\hmf(R,W)^\omega$. 

For a pair of objects $(R,W), (S,V) \in \LG$ the associated category of 1- and 2-morphisms is 
$$
\LG \big( (R,W), (S,V) \big) = \hmf( S\otimes_\C R, V-W)^\omega \, . 
$$
Horizontal composition in $\LG$ is the tensor product over the intermediate ring, i.\,e.~for 1-morphisms $(X,d_X) \in \LG((R_1,W_1), (R_2,W_2))$ and $(Y,d_Y) \in \LG((R_2,W_2), (R_3,W_3))$ it is given by 
$$
\big( Y \otimes_{R_2} X, d_Y \otimes 1 + 1 \otimes d_X \big)
\in 
\LG \big( (R_1, W_1), (R_3, W_3) \big) 
$$
which as explained in \cite{dm1102.2957}  (see also \cite[Sect.\,3.1]{khovhompaper}) can be explicitly computed by splitting an idempotent. In particular, $Y\otimes X$ is isomorphic to a finite-rank matrix factorisation over $R_3 \otimes_\C R_1$. 

The unit $I_W$ on $(R=\C[x_1,\ldots,x_n], W)$ is a deformation of the Koszul complex of $(x_1-x'_1), \ldots,(x_n-x'_n)$ over $R\otimes_\C R \cong \C[x_1,\ldots,x_n,x'_1,\ldots,x'_n]$ referred to as the stabilised diagonal in \cite{d0904.4713}. In the case $n=1$ the unit $I_W$ is simply the matrix factorisation
$$
\left( 
\C[x,x'] \oplus \C[x,x'], 
\begin{pmatrix}
0 &  x-x' \\ \tfrac{W(x)-W(x')}{x-x'} & 0
\end{pmatrix}
\right)
. 
$$
More details in $I_W$ as well as its left and right actions can e.\,g.~be found in \cite[Sect.\,2\,\&\,4]{cm1208.1481} and \cite[Sect.\,2\,\&\,A.1]{cr0909.4381}. 

\medskip

In our applications we will mostly be interested in \textsl{graded Landau-Ginzburg models}, which are described by a bicategory $\LGgr$. Its objects are also of the form $(\C[x_1,\ldots,x_n], W)$ where polynomials form a graded ring by assigning degrees $|x_i| \in \Q_+$ to the variables, and the isolated singularity $W \in \C[x_1,\ldots,x_n]$ must be quasi-homogeneous of degree 2: 
$$
W\big( \lambda^{|x_1|} x_1, \ldots, \lambda^{|x_n|} x_n \big) 
= 
\lambda^{2} \, W( x_1, \ldots, x_n )
\quad \text{for all } \lambda \in \C^\times \, . 
$$
The \textsl{central charge} of $(\C[x_1,\ldots,x_n], W)$ is the number
\be\label{eq:centralcharge}
c(W) = 3 \sum_{i=1}^n (1 - |x_i|) \, . 
\ee

	A 1-morphisms from $(\C[x_1,\ldots,x_n],W)$ to $(\C[z_1,\ldots,z_m],V)$ in $\LGgr$ is a matrix factorisation 
$(X,d_X)$ with the conditions that 
	the modules $X^0 = \bigoplus_{g\in \Q} X^0_q$ and $X^1= \bigoplus_{g\in \Q} X^1_q$ are $\Q$-graded, 
	acting with a polynomial of some degree amounts to an endomorphism of~$X$ of the same degree, $d_X$ has $\Q$-degree~1 on~$X$, and if we write $\{-\}$ for the shift in $\Q$-degree and if $X^i \cong \bigoplus_{q\in\Q} \Bbbk[x,z]\{q\}^{\oplus a_{i,q}}$ for $i \in \{0,1\}$, then $\{ q\in\Q \,|\, a_{i,q} \neq 0 \}$ must be a subset of $i + G_{V-W}$, 
	where 
	$
	G_{V-W} := \big\langle |x_1|, \dots, |x_n|, |z_1|, \dots, |z_m| \big\rangle \subset \Q 
	$ 
	and $G_0 := \Z$.  
Such graded matrix factorisations are the objects of triangulated categories $\hmfgr(R,W)$ whose morphisms by definition are maps in $\hmf(R,W)$ which have $\Q$-degree zero. We write $[-]$ for the $\Z_2$-grading 
	shift. 

We have 
$$
\LGgr \big( (R,W), (S,V) \big) = \hmfgr( S\otimes_\C R, V-W)^\omega \, , 
$$
and also the remainder of the construction of $\LGgr$ parallels that of $\LG$. 

\bigskip

The following is the main result of \cite{cm1208.1481} (see also \cite{bfk1105.3177v3} and \cite[Sect.\,2.2]{CRCR}). Crucially, quantum dimensions in $\LGgr$ are explicitly and easily computable (which is often not the case in other bicategories): 

\begin{theorem}
$\LGgr$ has adjoints. For a matrix factorisation $(X,d_X) : (\C[x_1,\ldots,x_n],W) \rightarrow (\C[z_1,\ldots,z_m],V)$ in $\LGgr$ we have 
$$
{}^\dagger X = X^\vee[m]\{\tfrac{1}{3}\, c(V)\} \, , \quad 
X^\dagger = X^\vee[n]\{\tfrac{1}{3}\, c(W)\} \, , 
$$
and for ambidextrous~$X$ (i.\,e.~iff $m=n \mod 2$ and $c(V) = c(W)$) we have
\begin{align}
\dim_{\textrm{l}}(X) 
& = 
(-1)^{\binom{n+1}{2}}\Res \left[ \frac{ \str\big( \partial_{x_1} d_{X}\ldots \partial_{x_n} d_{X} \,  \partial_{z_1} d_{X}\ldots \partial_{z_m} d_{X}\big) \operatorname{d}\! z}{\partial_{z_1} V, \ldots, \partial_{z_m} V} \right] , 
\nonumber
\\
\dim_{\textrm{r}}(X) 
& =
(-1)^{\binom{m+1}{2}}\Res \left[ \frac{ \str\big( \partial_{x_1} d_{X}\ldots \partial_{x_n} d_{X} \,  \partial_{z_1} d_{X}\ldots \partial_{z_m} d_{X}\big) \operatorname{d}\! x}{\partial_{x_1} W, \ldots, \partial_{x_n} W} \right] .
\nonumber
\end{align}
\end{theorem}

\subsubsection{ADE type orbifold equivalences}\label{subsubsec:ADEorbi}

Simple singularities fall into the following ADE classification: 
\begin{align}\label{eq:simplesing}
W^{(\mathrm{A}_{d-1})} &= x^{d} + y^{2} + z^2 && (|x| = \tfrac{2}{d} \, , \;\; |y| = |z| = 1) &
\nonumber\\
W^{(\mathrm{D}_{d+1})} &=  x^{d} + x y^{2} + z^2 && (|x| = \tfrac{2}{d} \, , \;\; |y| = 1 - \tfrac{1}{d} \, , \;\; |z| = 1) &
\nonumber\\
W^{(\mathrm{E}_6)} &= x^{3} + y^{4} + z^2 && (|x| = \tfrac{2}{3} \, , \;\; |y| = \tfrac{1}{2} \, , \;\; |z| = 1) & 
\\
W^{(\mathrm{E}_7)} &= x^{3} + x y^{3} + z^2 && (|x| = \tfrac{2}{3} \, , \;\; |y| = \tfrac{4}{9} \, , \;\; |z| = 1) & 
\nonumber\\
W^{(\mathrm{E}_8)} &= x^{3} + y^{5} + z^2 && (|x| = \tfrac{2}{3} \, , \;\; |y| = \tfrac{2}{5} \, , \;\; |z| = 1) & 
\nonumber
\end{align}
As elements of $R = \C[x,y,z]$ we consider these polynomials as objects in $\LGgr$. 

The main result of \cite{CRCR} is that two simple singularities are orbifold equivalent if and only if they have the same central charge~\eqref{eq:centralcharge}. To state it explicitly, we recall from \cite{br0707.0922} that there are so-called (graded) \textsl{permutation matrix factorisations} $P_S$ of $u'^d - u^d$ for every subset $S \subset \Z_d$: these are rank-1 factorisations on $\C[u,u'] \oplus \C[u,u']$ with 
$$
d_{P_S} = 
\begin{pmatrix}
0 & \prod_{l \in S} (u' - \zeta_d^l u) 
\\
\prod_{l \in S^{\textrm{c}}} (u' - \zeta_d^l u) & 0
\end{pmatrix}
$$
where $\zeta_d = \E^{2 \pi \I/d}$ and $S^{\textrm{c}}$ denotes the complement of~$S$ in $\Z_d$. Permutation matrix factorisations are well-understood, see e.\,g.~\cite[Sect.\,3.3]{cr1006.5609} and \cite[Sect.\,3.2]{DRCR}. Note in particular that $P_{\{0\}} = I_{u^d}$, and that $P_{\{d/2\}}[1] \cong P_{\{0,1,\ldots,d-1\} \backslash \{d/2\}}$ has twisted differential 
$$
\begin{pmatrix}
0 &  u'+u \\ \tfrac{u'^d - u^d}{u'-u} & 0
\end{pmatrix}
.
$$
Hence tensoring with $P_{\{d/2\}}[1]$ simply acts as $u \mapsto -u'$. 

\begin{theorem}[\cite{CRCR}]\label{thm:ADEOE}
In $\LGgr$ there are orbifold equivalences
\begin{align}\label{eq:ADEOE}
& W^{(\mathrm{A}_{2d-1})} \sim W^{(\mathrm{D}_{d+1})}
\, , \quad 
\nonumber
\\
& W^{(\mathrm{A}_{11})} \sim W^{(\mathrm{E}_6)} 
\, , \quad
W^{(\mathrm{A}_{17})} \sim W^{(\mathrm{E}_7)} 
\, , \quad
W^{(\mathrm{A}_{29})} \sim W^{(\mathrm{E}_8)} 
\, .
\end{align}
These equivalences are presented in \cite[Sect.\,2.3]{CRCR} by explicitly known 1-morphisms~$X$ (with source $W^{(\mathrm{A}_{2d-1})}$) which are rank-2 matrix factorisations in all cases but for $W^{(\mathrm{A}_{29})} \sim W^{(\mathrm{E}_8)}$ where~$X$ is rank-4. 
\end{theorem}

Moreover, the separable Frobenius algebras $X^\dagger \otimes X$ are isomorphic to certain direct sums of permutation matrix factorisations, which in particular leads to the equivalences\footnote{Here we are using \eqref{eq:con-equi} with $\B = \LGgr$, $a = (\C[x,y,z], W^{(\mathrm{A}_{2d-1})})$, $c = (\C,0)$, and choices for~$b$ and~$d$ informed by~\eqref{eq:ADEOE}.}
\begin{align}
\hmfgr\! \big( R, W^{(\mathrm{E}_6)} \big) 
& \cong 
\modu\! \big( P_{\{0\}} \oplus P_{\{-3,-2,\ldots,3\}} \big) \, , 
\nonumber
\\
\hmfgr\! \big( R, W^{(\mathrm{E}_7)} \big) 
& \cong 
\modu\! \big( P_{\{0\}} \oplus P_{\{-4,-3,\ldots,4\}} \oplus P_{\{-8,-7,\ldots,8\}} \big) \, , 
\label{eq:Psums}
\\
\hmfgr\! \big( R, W^{(\mathrm{E}_8)} \big) 
& \cong 
\modu\! \big( P_{\{0\}} \oplus P_{\{-5,-4,\ldots,5\}} \oplus P_{\{-9,-8,\ldots,9\}} \oplus P_{\{-14,-13,\ldots,14\}} \big) 
\nonumber
\end{align}
where $R= \C[x,y,z]$, as wells as (conjecturally for all~$d$, but it has only been checked for $d \in \{ 2,3,\ldots, 10\}$)
$$
\hmfgr\! \big( R, W^{(\mathrm{D}_{d+1})} \big) 
\cong 
\modu\! \big( P_{\{0\}} \oplus P_{\{ d/2 \}}[1] \big) \, .
$$

For example, in the case $W^{(\mathrm{A}_{11})} \sim W^{(\mathrm{E}_6)}$ the rank-2 matrix factorisation~$X$ of $x'^3 + y'^4 + z'^2 - x^{12} - y^2 - z^2$ is written out in \cite[(2.26)\,\&\,(2.27)]{CRCR} (in terms of only four variables, where one disposes of the squares $z^2,z'^2$ using the equivalence called Kn\"orrer periodicity \cite{knoe1987}), and gives rise to the functor
$$
- \otimes X : 
\hmfgr\! \big( \C[x',y',z'], W^{(\mathrm{E}_6)} \big) 
\lra
\hmfgr\! \big( \C[x,y,z], W^{(\mathrm{A}_{11})} \big) 
\, .
$$
Computing $X^\dagger \otimes X$ using the methods of \cite{dm1102.2957} implemented in \cite{khovhompaper} one obtains a matrix factorisation of $x'^{12} - y'^2 - z'^2 - x^{12} - y^2 - z^2$ which under Kn\"orrer periodicity corresponds to the matrix factorisation $A = P_{\{0\}} \oplus P_{\{-3,-2,\ldots,3\}}$ of $x'^{12} - x^{12}$. The algebra~$A$ gives rise to the endofunctor $A\otimes-$ on $\hmfgr(\C[x], x^{12})$, and its module category is equivalent to $\hmfgr ( \C[x,y,z], W^{(\mathrm{E}_6)})$. 

\medskip

The action of the matrix factorisations~$X$ exhibiting the orbifold equivalences of Theorem~\ref{thm:ADEOE} can be computed explicitly. To state the results, we recall that for a simple singularity $W^{(\Gamma)}$ of ADE type~$\Gamma$, every object in the category $\hmfgr(\C[x,y,z], W^{(\Gamma)})$ is isomorphic to a direct sum of (shifts of) simple objects 
$$
T^{(\Gamma)}_j \quad \text{with } j \in \{ 1,2, \ldots, |\Gamma| \} 
$$
which are listed in \cite[Sect.\,5]{kst0511155}. In fact $T^{(\Gamma)}_j$ precisely corresponds to the $j$-th vertex of the Dynkin diagram of type~$\Gamma$ with the vertices labelled as in Figure~\ref{fig:DynkinDiagrams}. 

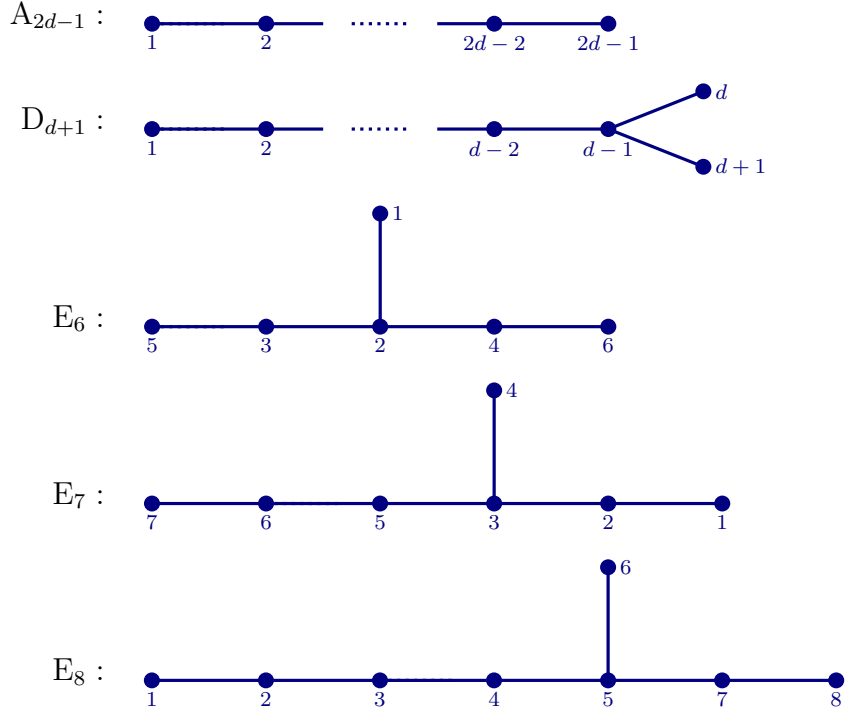
\begin{figure}[t]
\begin{align*}
\mathrm{A}_{2d-1}: \quad
&
\begin{tikzpicture}[very thick,scale=1.0,color=blue!50!black, baseline=0]
\draw[dotted] (0,0) -- (1,0);
\fill[] (0,0) circle (3.0pt) node[below] { {\scriptsize $1$} };
\fill[] (1.5,0) circle (3.0pt) node[below] { {\scriptsize $2$} };
\fill[] (4.5,0) circle (3.0pt) node[below] { {\scriptsize $2d-2$} };
\fill[] (6,0) circle (3.0pt) node[below] { {\scriptsize $2d-1$} };
\draw[-] (0,0) -- (1.5,0) -- (2.25,0); 
\draw[dotted] (2.625,0) -- (3.375,0);
\draw[-] (3.75,0) -- (4.5,0) -- (6,0); 
\end{tikzpicture}
\\
\mathrm{D}_{d+1}: \quad
&
\begin{tikzpicture}[very thick,scale=1.0,color=blue!50!black, baseline=0]
\draw[dotted] (0,0) -- (1,0);
\fill[] (0,0) circle (3.0pt) node[below] { {\scriptsize $1$} };
\fill[] (1.5,0) circle (3.0pt) node[below] { {\scriptsize $2$} };
\fill[] (4.5,0) circle (3.0pt) node[below] { {\scriptsize $d-2$} };
\fill[] (6,0) circle (3.0pt) node[below] { {\scriptsize $d-1$} };
\fill[] (7.25,0.5) circle (3.0pt) node[right] { {\scriptsize $d$} };
\fill[] (7.25,-0.5) circle (3.0pt) node[right] { {\scriptsize $d+1$} };
\draw[-] (0,0) -- (1.5,0) -- (2.25,0); 
\draw[dotted] (2.625,0) -- (3.375,0);
\draw[-] (3.75,0) -- (4.5,0) -- (6,0); 
\draw[-] (6,0) -- (7.25,-0.5); 
\draw[-] (6,0) -- (7.25,0.5); 
\end{tikzpicture}
\\
\mathrm{E}_{6}: \quad
&
\begin{tikzpicture}[very thick,scale=1.0,color=blue!50!black, baseline=0]
\draw[dotted] (0,0) -- (1,0);
\fill[] (0,0) circle (3.0pt) node[below] { {\scriptsize $5$} };
\fill[] (1.5,0) circle (3.0pt) node[below] { {\scriptsize $3$} };
\fill[] (3,0) circle (3.0pt) node[below] { {\scriptsize $2$} };
\fill[] (4.5,0) circle (3.0pt) node[below] { {\scriptsize $4$} };
\fill[] (6,0) circle (3.0pt) node[below] { {\scriptsize $6$} };
\fill[] (3,1.5) circle (3.0pt) node[right] { {\scriptsize $1$} };
\draw[-] (0,0) -- (6,0); 
\draw[-] (3,0) -- (3,1.5); 
\end{tikzpicture}
\\
\mathrm{E}_{7}: \quad
&
\begin{tikzpicture}[very thick,scale=1.0,color=blue!50!black, baseline=0]
\draw[dotted] (0,0) -- (1,0);
\fill[] (-1.5,0) circle (3.0pt) node[below] { {\scriptsize $7$} };
\fill[] (0,0) circle (3.0pt) node[below] { {\scriptsize $6$} };
\fill[] (1.5,0) circle (3.0pt) node[below] { {\scriptsize $5$} };
\fill[] (3,0) circle (3.0pt) node[below] { {\scriptsize $3$} };
\fill[] (4.5,0) circle (3.0pt) node[below] { {\scriptsize $2$} };
\fill[] (6,0) circle (3.0pt) node[below] { {\scriptsize $1$} };
\fill[] (3,1.5) circle (3.0pt) node[right] { {\scriptsize $4$} };
\draw[-] (-1.5,0) -- (6,0); 
\draw[-] (3,0) -- (3,1.5); 
\end{tikzpicture}
\\
\mathrm{E}_{8}: \quad
&
\begin{tikzpicture}[very thick,scale=1.0,color=blue!50!black, baseline=0]
\draw[dotted] (0,0) -- (1,0);
\fill[] (-3,0) circle (3.0pt) node[below] { {\scriptsize $1$} };
\fill[] (-1.5,0) circle (3.0pt) node[below] { {\scriptsize $2$} };
\fill[] (0,0) circle (3.0pt) node[below] { {\scriptsize $3$} };
\fill[] (1.5,0) circle (3.0pt) node[below] { {\scriptsize $4$} };
\fill[] (3,0) circle (3.0pt) node[below] { {\scriptsize $5$} };
\fill[] (4.5,0) circle (3.0pt) node[below] { {\scriptsize $7$} };
\fill[] (6,0) circle (3.0pt) node[below] { {\scriptsize $8$} };
\fill[] (3,1.5) circle (3.0pt) node[right] { {\scriptsize $6$} };
\draw[-] (-3,0) -- (6,0); 
\draw[-] (3,0) -- (3,1.5); 
\end{tikzpicture}
\end{align*}
\caption{ADE Dynkin diagrams with vertex label convention of \cite{kst0511155}} 
\label{fig:DynkinDiagrams} 
\end{figure}

\begin{proposition}\label{prop:Xaction}
The functors $-\otimes X$ of Theorem~\ref{thm:ADEOE} act as follows (up to isomorphism and shifts) on the simples 
$
T^{(\Gamma)}_j 
$: 
\begin{enumerate}
\item 
as a functor $\hmfgr(\C[x,y,z], W^{(\mathrm{D}_{d+1})}) \rightarrow \hmfgr(\C[x,y,z], W^{(\mathrm{A}_{2d-1})})$: 
\begin{align*}
T^{(\mathrm{D}_{d+1})}_1 & \longmapsto T^{(\mathrm{A}_{2d-1})}_1 \oplus T^{(\mathrm{A}_{2d-1})}_{2d-1} , 
\\
T^{(\mathrm{D}_{d+1})}_j & \longmapsto T^{(\mathrm{A}_{2d-1})}_j \oplus T^{(\mathrm{A}_{2d-1})}_{2d-j} 
\quad\text{for } j \in \{ 2,4,\ldots, d-2 \}  , 
\\
T^{(\mathrm{D}_{d+1})}_{j+1} & \longmapsto T^{(\mathrm{A}_{2d-1})}_{j-1} \oplus T^{(\mathrm{A}_{2d-1})}_{2d-j+1} 
\quad\text{for } j \in \{ 2,4,\ldots, d-2 \}  ,
\\
T^{(\mathrm{D}_{d+1})}_d & \longmapsto T^{(\mathrm{A}_{2d-1})}_d  , 
\\
T^{(\mathrm{D}_{d+1})}_{d+1} & \longmapsto T^{(\mathrm{A}_{2d-1})}_d  , 
\end{align*}
for at least $d \in \{ 2,3, \ldots , 42 \}$, 
\item 
as a functor $\hmfgr(\C[x,y,z], W^{(\mathrm{E}_{6})}) \rightarrow \hmfgr(\C[x,y,z], W^{(\mathrm{A}_{11})})$: 
\begin{align*}
T^{(\mathrm{E}_{6})}_1 & \longmapsto T^{(\mathrm{A}_{11})}_{4} \oplus T^{(\mathrm{A}_{11})}_{8} , 
\\
T^{(\mathrm{E}_{6})}_2 & \longmapsto T^{(\mathrm{A}_{11})}_{3} \oplus T^{(\mathrm{A}_{11})}_{5} \oplus T^{(\mathrm{A}_{11})}_{7} \oplus T^{(\mathrm{A}_{11})}_{9} , 
\\
T^{(\mathrm{E}_{6})}_3 & \longmapsto T^{(\mathrm{A}_{11})}_{2} \oplus T^{(\mathrm{A}_{11})}_{6} \oplus T^{(\mathrm{A}_{11})}_{8} , 
\\
T^{(\mathrm{E}_{6})}_4 & \longmapsto T^{(\mathrm{A}_{11})}_{4} \oplus T^{(\mathrm{A}_{11})}_{6} \oplus T^{(\mathrm{A}_{11})}_{10} , 
\\
T^{(\mathrm{E}_{6})}_5 & \longmapsto T^{(\mathrm{A}_{11})}_{1} \oplus T^{(\mathrm{A}_{11})}_{7} , 
\\
T^{(\mathrm{E}_{6})}_6 & \longmapsto T^{(\mathrm{A}_{11})}_{5} \oplus T^{(\mathrm{A}_{11})}_{11} , 
\end{align*}
\item 
as a functor $\hmfgr(\C[x,y,z], W^{(\mathrm{E}_{7})}) \rightarrow \hmfgr(\C[x,y,z], W^{(\mathrm{A}_{17})})$: 
\begin{align*}
T^{(\mathrm{E}_{7})}_1 & \longmapsto T^{(\mathrm{A}_{17})}_{6} \oplus T^{(\mathrm{A}_{17})}_{12} , 
\\
T^{(\mathrm{E}_{7})}_2 & \longmapsto T^{(\mathrm{A}_{17})}_{5} \oplus T^{(\mathrm{A}_{17})}_{7} \oplus T^{(\mathrm{A}_{17})}_{11} \oplus T^{(\mathrm{A}_{17})}_{13} , 
\\
T^{(\mathrm{E}_{7})}_3 & \longmapsto T^{(\mathrm{A}_{17})}_{4} \oplus T^{(\mathrm{A}_{17})}_{6} \oplus T^{(\mathrm{A}_{17})}_{8} \oplus T^{(\mathrm{A}_{17})}_{10} \oplus T^{(\mathrm{A}_{17})}_{12} \oplus T^{(\mathrm{A}_{17})}_{14} , 
\\
T^{(\mathrm{E}_{7})}_4 & \longmapsto T^{(\mathrm{A}_{17})}_{5} \oplus T^{(\mathrm{A}_{17})}_{9} \oplus T^{(\mathrm{A}_{17})}_{13} , 
\\
T^{(\mathrm{E}_{7})}_5 & \longmapsto T^{(\mathrm{A}_{17})}_{3} \oplus T^{(\mathrm{A}_{17})}_{7} \oplus T^{(\mathrm{A}_{17})}_{9} \oplus T^{(\mathrm{A}_{17})}_{11} \oplus T^{(\mathrm{A}_{17})}_{15} , 
\\
T^{(\mathrm{E}_{7})}_6 & \longmapsto T^{(\mathrm{A}_{17})}_{2} \oplus T^{(\mathrm{A}_{17})}_{8} \oplus T^{(\mathrm{A}_{17})}_{10} \oplus T^{(\mathrm{A}_{17})}_{16} , 
\\
T^{(\mathrm{E}_{7})}_7 & \longmapsto T^{(\mathrm{A}_{17})}_{1} \oplus T^{(\mathrm{A}_{17})}_{9} \oplus T^{(\mathrm{A}_{17})}_{17} . 
\end{align*}
\end{enumerate}
\end{proposition}

\begin{proof}
Direct computation using the methods implemented in Singular in \cite{khovhompaper}. 
\end{proof}

The action of the 1-morphism~$X$ exhibiting the orbifold equivalence $W^{(\mathrm{A}_{29})} \sim W^{(\mathrm{E}_8)} $ is missing in Proposition~\ref{prop:Xaction}. Computing the matrix factorisations $T^{(\mathrm{E}_{8})}_j \otimes X$ using the methods of \cite{khovhompaper} is easy and fast in this case too,  but finding explicit isomorphisms to direct sums of A-type simples is left as an exercise to the enthusiastically devoted reader. 

One can also compute the action of the Frobenius algebras $A = X^\dagger \otimes X$: 

\begin{proposition}\label{prop:Aaction}
Let~$X$ exhibit one of the orbifold equivalences $W^{(\mathrm{A}_{11})} \sim W^{(\mathrm{E}_6)}$, $W^{(\mathrm{A}_{17})} \sim W^{(\mathrm{E}_7)}$ or $W^{(\mathrm{A}_{29})} \sim W^{(\mathrm{E}_8)}$, and set $A = X^\dagger \otimes X$. Then the functors $A\otimes -$ act as follows (up to isomorphism and shift): 
\begin{align*}
T^{(\mathrm{A}_{11})}_j & \lmt \sum_{i=1}^{11} M_{ij}^{(\mathrm{A}_{11})} \, T^{(\mathrm{A}_{11})}_i
\quad\text{on } \hmfgr\!\big(\C[x,y,z], W^{(\mathrm{A}_{11})}\big) \, , 
\\
T^{(\mathrm{A}_{17})}_j & \lmt \sum_{i=1}^{17} M_{ij}^{(\mathrm{A}_{17})} \, T^{(\mathrm{A}_{17})}_i
\quad\text{on } \hmfgr\!\big(\C[x,y,z], W^{(\mathrm{A}_{17})}\big) \, , 
\\
T^{(\mathrm{A}_{29})}_j & \lmt \sum_{i=1}^{29} M_{ij}^{(\mathrm{A}_{29})} \, T^{(\mathrm{A}_{29})}_i
\quad\text{on } \hmfgr\!\big(\C[x,y,z], W^{(\mathrm{A}_{29})}\big) 
\end{align*}
where 
\begin{align*}
M^{(\mathrm{A}_{11})}
& = 
{\tiny 
\begin{pmatrix}
       1 &\!\! 0 &\!\! 0 &\!\! 0 &\!\! 0 &\!\! 0 &\!\! 1 &\!\! 0 &\!\! 0 &\!\! 0 &\!\! 0 \\
       0 &\!\! 1 &\!\! 0 &\!\! 0 &\!\! 0 &\!\! 1 &\!\! 0 &\!\! 1 &\!\! 0 &\!\! 0 &\!\! 0 \\
       0 &\!\! 0 &\!\! 1 &\!\! 0 &\!\! 1 &\!\! 0 &\!\! 1 &\!\! 0 &\!\! 1 &\!\! 0 &\!\! 0 \\
       0 &\!\! 0 &\!\! 0 &\!\! 2 &\!\! 0 &\!\! 1 &\!\! 0 &\!\! 1 &\!\! 0 &\!\! 1 &\!\! 0 \\
       0 &\!\! 0 &\!\! 1 &\!\! 0 &\!\! 2 &\!\! 0 &\!\! 1 &\!\! 0 &\!\! 1 &\!\! 0 &\!\! 1 \\
       0 &\!\! 1 &\!\! 0 &\!\! 1 &\!\! 0 &\!\! 2 &\!\! 0 &\!\! 1 &\!\! 0 &\!\! 1 &\!\! 0 \\
       1 &\!\! 0 &\!\! 1 &\!\! 0 &\!\! 1 &\!\! 0 &\!\! 2 &\!\! 0 &\!\! 1 &\!\! 0 &\!\! 0 \\
       0 &\!\! 1 &\!\! 0 &\!\! 1 &\!\! 0 &\!\! 1 &\!\! 0 &\!\! 2 &\!\! 0 &\!\! 0 &\!\! 0 \\
       0 &\!\! 0 &\!\! 1 &\!\! 0 &\!\! 1 &\!\! 0 &\!\! 1 &\!\! 0 &\!\! 1 &\!\! 0 &\!\! 0 \\
       0 &\!\! 0 &\!\! 0 &\!\! 1 &\!\! 0 &\!\! 1 &\!\! 0 &\!\! 0 &\!\! 0 &\!\! 1 &\!\! 0 \\
       0 &\!\! 0 &\!\! 0 &\!\! 0 &\!\! 1 &\!\! 0 &\!\! 0 &\!\! 0 &\!\! 0 &\!\! 0 &\!\! 1
\end{pmatrix}
} ,
\\
M^{(\mathrm{A}_{17})}
& = 
{\tiny 
\begin{pmatrix}
        1 &\!\! 0 &\!\! 0 &\!\! 0 &\!\! 0 &\!\! 0 &\!\! 0 &\!\! 0 &\!\! 1 &\!\! 0 &\!\! 0 &\!\! 0 &\!\! 0 &\!\! 0 &\!\!
          0 &\!\! 0 &\!\! 1  \\
        0 &\!\! 1 &\!\! 0 &\!\! 0 &\!\! 0 &\!\! 0 &\!\! 0 &\!\! 1 &\!\! 0 &\!\! 1 &\!\! 0 &\!\! 0 &\!\! 0 &\!\! 0 &\!\!
          0 &\!\! 1 &\!\! 0  \\
        0 &\!\! 0 &\!\! 1 &\!\! 0 &\!\! 0 &\!\! 0 &\!\! 1 &\!\! 0 &\!\! 1 &\!\! 0 &\!\! 1 &\!\! 0 &\!\! 0 &\!\! 0 &\!\!
          1 &\!\! 0 &\!\! 0  \\
        0 &\!\! 0 &\!\! 0 &\!\! 1 &\!\! 0 &\!\! 1 &\!\! 0 &\!\! 1 &\!\! 0 &\!\! 1 &\!\! 0 &\!\! 1 &\!\! 0 &\!\! 1 &\!\!
          0 &\!\! 0 &\!\! 0  \\
        0 &\!\! 0 &\!\! 0 &\!\! 0 &\!\! 2 &\!\! 0 &\!\! 1 &\!\! 0 &\!\! 1 &\!\! 0 &\!\! 1 &\!\! 0 &\!\! 2 &\!\! 0 &\!\!
          0 &\!\! 0 &\!\! 0  \\
        0 &\!\! 0 &\!\! 0 &\!\! 1 &\!\! 0 &\!\! 2 &\!\! 0 &\!\! 1 &\!\! 0 &\!\! 1 &\!\! 0 &\!\! 2 &\!\! 0 &\!\! 1 &\!\!
          0 &\!\! 0 &\!\! 0  \\
        0 &\!\! 0 &\!\! 1 &\!\! 0 &\!\! 1 &\!\! 0 &\!\! 2 &\!\! 0 &\!\! 1 &\!\! 0 &\!\! 2 &\!\! 0 &\!\! 1 &\!\! 0 &\!\!
          1 &\!\! 0 &\!\! 0  \\
        0 &\!\! 1 &\!\! 0 &\!\! 1 &\!\! 0 &\!\! 1 &\!\! 0 &\!\! 2 &\!\! 0 &\!\! 2 &\!\! 0 &\!\! 1 &\!\! 0 &\!\! 1 &\!\!
          0 &\!\! 1 &\!\! 0  \\
        1 &\!\! 0 &\!\! 1 &\!\! 0 &\!\! 1 &\!\! 0 &\!\! 1 &\!\! 0 &\!\! 3 &\!\! 0 &\!\! 1 &\!\! 0 &\!\! 1 &\!\! 0 &\!\!
          1 &\!\! 0 &\!\! 1  \\
        0 &\!\! 1 &\!\! 0 &\!\! 1 &\!\! 0 &\!\! 1 &\!\! 0 &\!\! 2 &\!\! 0 &\!\! 2 &\!\! 0 &\!\! 1 &\!\! 0 &\!\! 1 &\!\!
          0 &\!\! 1 &\!\! 0  \\
        0 &\!\! 0 &\!\! 1 &\!\! 0 &\!\! 1 &\!\! 0 &\!\! 2 &\!\! 0 &\!\! 1 &\!\! 0 &\!\! 2 &\!\! 0 &\!\! 1 &\!\! 0 &\!\!
          1 &\!\! 0 &\!\! 0  \\
        0 &\!\! 0 &\!\! 0 &\!\! 1 &\!\! 0 &\!\! 2 &\!\! 0 &\!\! 1 &\!\! 0 &\!\! 1 &\!\! 0 &\!\! 2 &\!\! 0 &\!\! 1 &\!\!
          0 &\!\! 0 &\!\! 0  \\
        0 &\!\! 0 &\!\! 0 &\!\! 0 &\!\! 2 &\!\! 0 &\!\! 1 &\!\! 0 &\!\! 1 &\!\! 0 &\!\! 1 &\!\! 0 &\!\! 2 &\!\! 0 &\!\!
          0 &\!\! 0 &\!\! 0  \\
        0 &\!\! 0 &\!\! 0 &\!\! 1 &\!\! 0 &\!\! 1 &\!\! 0 &\!\! 1 &\!\! 0 &\!\! 1 &\!\! 0 &\!\! 1 &\!\! 0 &\!\! 1 &\!\!
          0 &\!\! 0 &\!\! 0  \\
        0 &\!\! 0 &\!\! 1 &\!\! 0 &\!\! 0 &\!\! 0 &\!\! 1 &\!\! 0 &\!\! 1 &\!\! 0 &\!\! 1 &\!\! 0 &\!\! 0 &\!\! 0 &\!\!
          1 &\!\! 0 &\!\! 0  \\
        0 &\!\! 1 &\!\! 0 &\!\! 0 &\!\! 0 &\!\! 0 &\!\! 0 &\!\! 1 &\!\! 0 &\!\! 1 &\!\! 0 &\!\! 0 &\!\! 0 &\!\! 0 &\!\!
          0 &\!\! 1 &\!\! 0  \\
        1 &\!\! 0 &\!\! 0 &\!\! 0 &\!\! 0 &\!\! 0 &\!\! 0 &\!\! 0 &\!\! 1 &\!\! 0 &\!\! 0 &\!\! 0 &\!\! 0 &\!\! 0 &\!\!
          0 &\!\! 0 &\!\! 1
\end{pmatrix}
} ,
\\
M^{(\mathrm{A}_{29})}
& = 
{\tiny
\begin{pmatrix}
        1 &\!\! 0 &\!\! 0 &\!\! 0 &\!\! 0 &\!\! 0 &\!\! 0 &\!\! 0 &\!\! 0 &\!\! 0 &\!\! 1 &\!\! 0 &\!\! 0 &\!\! 0 &\!\!
          0 &\!\! 0 &\!\! 0 &\!\! 0 &\!\! 1 &\!\! 0 &\!\! 0 &\!\! 0 &\!\! 0 &\!\! 0 &\!\! 0 &\!\! 0 &\!\! 0 &\!\! 
         0 &\!\! 1  \\
        0 &\!\! 1 &\!\! 0 &\!\! 0 &\!\! 0 &\!\! 0 &\!\! 0 &\!\! 0 &\!\! 0 &\!\! 1 &\!\! 0 &\!\! 1 &\!\! 0 &\!\! 0 &\!\!
          0 &\!\! 0 &\!\! 0 &\!\! 1 &\!\! 0 &\!\! 1 &\!\! 0 &\!\! 0 &\!\! 0 &\!\! 0 &\!\! 0 &\!\! 0 &\!\! 0 &\!\! 
         1 &\!\! 0  \\
        0 &\!\! 0 &\!\! 1 &\!\! 0 &\!\! 0 &\!\! 0 &\!\! 0 &\!\! 0 &\!\! 1 &\!\! 0 &\!\! 1 &\!\! 0 &\!\! 1 &\!\! 0 &\!\!
          0 &\!\! 0 &\!\! 1 &\!\! 0 &\!\! 1 &\!\! 0 &\!\! 1 &\!\! 0 &\!\! 0 &\!\! 0 &\!\! 0 &\!\! 0 &\!\! 1 &\!\! 
         0 &\!\! 0  \\
        0 &\!\! 0 &\!\! 0 &\!\! 1 &\!\! 0 &\!\! 0 &\!\! 0 &\!\! 1 &\!\! 0 &\!\! 1 &\!\! 0 &\!\! 1 &\!\! 0 &\!\! 1 &\!\!
          0 &\!\! 1 &\!\! 0 &\!\! 1 &\!\! 0 &\!\! 1 &\!\! 0 &\!\! 1 &\!\! 0 &\!\! 0 &\!\! 0 &\!\! 1 &\!\! 0 &\!\! 
         0 &\!\! 0  \\
        0 &\!\! 0 &\!\! 0 &\!\! 0 &\!\! 1 &\!\! 0 &\!\! 1 &\!\! 0 &\!\! 1 &\!\! 0 &\!\! 1 &\!\! 0 &\!\! 1 &\!\! 0 &\!\!
          2 &\!\! 0 &\!\! 1 &\!\! 0 &\!\! 1 &\!\! 0 &\!\! 1 &\!\! 0 &\!\! 1 &\!\! 0 &\!\! 1 &\!\! 0 &\!\! 0 &\!\! 
         0 &\!\! 0  \\
        0 &\!\! 0 &\!\! 0 &\!\! 0 &\!\! 0 &\!\! 2 &\!\! 0 &\!\! 1 &\!\! 0 &\!\! 1 &\!\! 0 &\!\! 1 &\!\! 0 &\!\! 2 &\!\!
          0 &\!\! 2 &\!\! 0 &\!\! 1 &\!\! 0 &\!\! 1 &\!\! 0 &\!\! 1 &\!\! 0 &\!\! 2 &\!\! 0 &\!\! 0 &\!\! 0 &\!\! 
         0 &\!\! 0  \\
        0 &\!\! 0 &\!\! 0 &\!\! 0 &\!\! 1 &\!\! 0 &\!\! 2 &\!\! 0 &\!\! 1 &\!\! 0 &\!\! 1 &\!\! 0 &\!\! 2 &\!\! 0 &\!\!
          2 &\!\! 0 &\!\! 2 &\!\! 0 &\!\! 1 &\!\! 0 &\!\! 1 &\!\! 0 &\!\! 2 &\!\! 0 &\!\! 1 &\!\! 0 &\!\! 0 &\!\! 
         0 &\!\! 0  \\
        0 &\!\! 0 &\!\! 0 &\!\! 1 &\!\! 0 &\!\! 1 &\!\! 0 &\!\! 2 &\!\! 0 &\!\! 1 &\!\! 0 &\!\! 2 &\!\! 0 &\!\! 2 &\!\!
          0 &\!\! 2 &\!\! 0 &\!\! 2 &\!\! 0 &\!\! 1 &\!\! 0 &\!\! 2 &\!\! 0 &\!\! 1 &\!\! 0 &\!\! 1 &\!\! 0 &\!\! 
         0 &\!\! 0  \\
        0 &\!\! 0 &\!\! 1 &\!\! 0 &\!\! 1 &\!\! 0 &\!\! 1 &\!\! 0 &\!\! 2 &\!\! 0 &\!\! 2 &\!\! 0 &\!\! 2 &\!\! 0 &\!\!
          2 &\!\! 0 &\!\! 2 &\!\! 0 &\!\! 2 &\!\! 0 &\!\! 2 &\!\! 0 &\!\! 1 &\!\! 0 &\!\! 1 &\!\! 0 &\!\! 1 &\!\! 
         0 &\!\! 0  \\
        0 &\!\! 1 &\!\! 0 &\!\! 1 &\!\! 0 &\!\! 1 &\!\! 0 &\!\! 1 &\!\! 0 &\!\! 3 &\!\! 0 &\!\! 2 &\!\! 0 &\!\! 2 &\!\!
          0 &\!\! 2 &\!\! 0 &\!\! 2 &\!\! 0 &\!\! 3 &\!\! 0 &\!\! 1 &\!\! 0 &\!\! 1 &\!\! 0 &\!\! 1 &\!\! 0 &\!\! 
         1 &\!\! 0  \\
        1 &\!\! 0 &\!\! 1 &\!\! 0 &\!\! 1 &\!\! 0 &\!\! 1 &\!\! 0 &\!\! 2 &\!\! 0 &\!\! 3 &\!\! 0 &\!\! 2 &\!\! 0 &\!\!
          2 &\!\! 0 &\!\! 2 &\!\! 0 &\!\! 3 &\!\! 0 &\!\! 2 &\!\! 0 &\!\! 1 &\!\! 0 &\!\! 1 &\!\! 0 &\!\! 1 &\!\! 
         0 &\!\! 1  \\
        0 &\!\! 1 &\!\! 0 &\!\! 1 &\!\! 0 &\!\! 1 &\!\! 0 &\!\! 2 &\!\! 0 &\!\! 2 &\!\! 0 &\!\! 3 &\!\! 0 &\!\! 2 &\!\!
          0 &\!\! 2 &\!\! 0 &\!\! 3 &\!\! 0 &\!\! 2 &\!\! 0 &\!\! 2 &\!\! 0 &\!\! 1 &\!\! 0 &\!\! 1 &\!\! 0 &\!\! 
         1 &\!\! 0  \\
        0 &\!\! 0 &\!\! 1 &\!\! 0 &\!\! 1 &\!\! 0 &\!\! 2 &\!\! 0 &\!\! 2 &\!\! 0 &\!\! 2 &\!\! 0 &\!\! 3 &\!\! 0 &\!\!
          2 &\!\! 0 &\!\! 3 &\!\! 0 &\!\! 2 &\!\! 0 &\!\! 2 &\!\! 0 &\!\! 2 &\!\! 0 &\!\! 1 &\!\! 0 &\!\! 1 &\!\! 
         0 &\!\! 0  \\
        0 &\!\! 0 &\!\! 0 &\!\! 1 &\!\! 0 &\!\! 2 &\!\! 0 &\!\! 2 &\!\! 0 &\!\! 2 &\!\! 0 &\!\! 2 &\!\! 0 &\!\! 3 &\!\!
          0 &\!\! 3 &\!\! 0 &\!\! 2 &\!\! 0 &\!\! 2 &\!\! 0 &\!\! 2 &\!\! 0 &\!\! 2 &\!\! 0 &\!\! 1 &\!\! 0 &\!\! 
         0 &\!\! 0  \\
        0 &\!\! 0 &\!\! 0 &\!\! 0 &\!\! 2 &\!\! 0 &\!\! 2 &\!\! 0 &\!\! 2 &\!\! 0 &\!\! 2 &\!\! 0 &\!\! 2 &\!\! 0 &\!\!
          4 &\!\! 0 &\!\! 2 &\!\! 0 &\!\! 2 &\!\! 0 &\!\! 2 &\!\! 0 &\!\! 2 &\!\! 0 &\!\! 2 &\!\! 0 &\!\! 0 &\!\! 
         0 &\!\! 0  \\
        0 &\!\! 0 &\!\! 0 &\!\! 1 &\!\! 0 &\!\! 2 &\!\! 0 &\!\! 2 &\!\! 0 &\!\! 2 &\!\! 0 &\!\! 2 &\!\! 0 &\!\! 3 &\!\!
          0 &\!\! 3 &\!\! 0 &\!\! 2 &\!\! 0 &\!\! 2 &\!\! 0 &\!\! 2 &\!\! 0 &\!\! 2 &\!\! 0 &\!\! 1 &\!\! 0 &\!\! 
         0 &\!\! 0  \\
        0 &\!\! 0 &\!\! 1 &\!\! 0 &\!\! 1 &\!\! 0 &\!\! 2 &\!\! 0 &\!\! 2 &\!\! 0 &\!\! 2 &\!\! 0 &\!\! 3 &\!\! 0 &\!\!
          2 &\!\! 0 &\!\! 3 &\!\! 0 &\!\! 2 &\!\! 0 &\!\! 2 &\!\! 0 &\!\! 2 &\!\! 0 &\!\! 1 &\!\! 0 &\!\! 1 &\!\! 
         0 &\!\! 0  \\
        0 &\!\! 1 &\!\! 0 &\!\! 1 &\!\! 0 &\!\! 1 &\!\! 0 &\!\! 2 &\!\! 0 &\!\! 2 &\!\! 0 &\!\! 3 &\!\! 0 &\!\! 2 &\!\!
          0 &\!\! 2 &\!\! 0 &\!\! 3 &\!\! 0 &\!\! 2 &\!\! 0 &\!\! 2 &\!\! 0 &\!\! 1 &\!\! 0 &\!\! 1 &\!\! 0 &\!\! 
         1 &\!\! 0  \\
        1 &\!\! 0 &\!\! 1 &\!\! 0 &\!\! 1 &\!\! 0 &\!\! 1 &\!\! 0 &\!\! 2 &\!\! 0 &\!\! 3 &\!\! 0 &\!\! 2 &\!\! 0 &\!\!
          2 &\!\! 0 &\!\! 2 &\!\! 0 &\!\! 3 &\!\! 0 &\!\! 2 &\!\! 0 &\!\! 1 &\!\! 0 &\!\! 1 &\!\! 0 &\!\! 1 &\!\! 
         0 &\!\! 1  \\
        0 &\!\! 1 &\!\! 0 &\!\! 1 &\!\! 0 &\!\! 1 &\!\! 0 &\!\! 1 &\!\! 0 &\!\! 3 &\!\! 0 &\!\! 2 &\!\! 0 &\!\! 2 &\!\!
          0 &\!\! 2 &\!\! 0 &\!\! 2 &\!\! 0 &\!\! 3 &\!\! 0 &\!\! 1 &\!\! 0 &\!\! 1 &\!\! 0 &\!\! 1 &\!\! 0 &\!\! 
         1 &\!\! 0  \\
        0 &\!\! 0 &\!\! 1 &\!\! 0 &\!\! 1 &\!\! 0 &\!\! 1 &\!\! 0 &\!\! 2 &\!\! 0 &\!\! 2 &\!\! 0 &\!\! 2 &\!\! 0 &\!\!
          2 &\!\! 0 &\!\! 2 &\!\! 0 &\!\! 2 &\!\! 0 &\!\! 2 &\!\! 0 &\!\! 1 &\!\! 0 &\!\! 1 &\!\! 0 &\!\! 1 &\!\! 
         0 &\!\! 0  \\
        0 &\!\! 0 &\!\! 0 &\!\! 1 &\!\! 0 &\!\! 1 &\!\! 0 &\!\! 2 &\!\! 0 &\!\! 1 &\!\! 0 &\!\! 2 &\!\! 0 &\!\! 2 &\!\!
          0 &\!\! 2 &\!\! 0 &\!\! 2 &\!\! 0 &\!\! 1 &\!\! 0 &\!\! 2 &\!\! 0 &\!\! 1 &\!\! 0 &\!\! 1 &\!\! 0 &\!\! 
         0 &\!\! 0  \\
        0 &\!\! 0 &\!\! 0 &\!\! 0 &\!\! 1 &\!\! 0 &\!\! 2 &\!\! 0 &\!\! 1 &\!\! 0 &\!\! 1 &\!\! 0 &\!\! 2 &\!\! 0 &\!\!
          2 &\!\! 0 &\!\! 2 &\!\! 0 &\!\! 1 &\!\! 0 &\!\! 1 &\!\! 0 &\!\! 2 &\!\! 0 &\!\! 1 &\!\! 0 &\!\! 0 &\!\! 
         0 &\!\! 0  \\
        0 &\!\! 0 &\!\! 0 &\!\! 0 &\!\! 0 &\!\! 2 &\!\! 0 &\!\! 1 &\!\! 0 &\!\! 1 &\!\! 0 &\!\! 1 &\!\! 0 &\!\! 2 &\!\!
          0 &\!\! 2 &\!\! 0 &\!\! 1 &\!\! 0 &\!\! 1 &\!\! 0 &\!\! 1 &\!\! 0 &\!\! 2 &\!\! 0 &\!\! 0 &\!\! 0 &\!\! 
         0 &\!\! 0  \\
        0 &\!\! 0 &\!\! 0 &\!\! 0 &\!\! 1 &\!\! 0 &\!\! 1 &\!\! 0 &\!\! 1 &\!\! 0 &\!\! 1 &\!\! 0 &\!\! 1 &\!\! 0 &\!\!
          2 &\!\! 0 &\!\! 1 &\!\! 0 &\!\! 1 &\!\! 0 &\!\! 1 &\!\! 0 &\!\! 1 &\!\! 0 &\!\! 1 &\!\! 0 &\!\! 0 &\!\! 
         0 &\!\! 0  \\
        0 &\!\! 0 &\!\! 0 &\!\! 1 &\!\! 0 &\!\! 0 &\!\! 0 &\!\! 1 &\!\! 0 &\!\! 1 &\!\! 0 &\!\! 1 &\!\! 0 &\!\! 1 &\!\!
          0 &\!\! 1 &\!\! 0 &\!\! 1 &\!\! 0 &\!\! 1 &\!\! 0 &\!\! 1 &\!\! 0 &\!\! 0 &\!\! 0 &\!\! 1 &\!\! 0 &\!\! 
         0 &\!\! 0  \\
        0 &\!\! 0 &\!\! 1 &\!\! 0 &\!\! 0 &\!\! 0 &\!\! 0 &\!\! 0 &\!\! 1 &\!\! 0 &\!\! 1 &\!\! 0 &\!\! 1 &\!\! 0 &\!\!
          0 &\!\! 0 &\!\! 1 &\!\! 0 &\!\! 1 &\!\! 0 &\!\! 1 &\!\! 0 &\!\! 0 &\!\! 0 &\!\! 0 &\!\! 0 &\!\! 1 &\!\! 
         0 &\!\! 0  \\
        0 &\!\! 1 &\!\! 0 &\!\! 0 &\!\! 0 &\!\! 0 &\!\! 0 &\!\! 0 &\!\! 0 &\!\! 1 &\!\! 0 &\!\! 1 &\!\! 0 &\!\! 0 &\!\!
          0 &\!\! 0 &\!\! 0 &\!\! 1 &\!\! 0 &\!\! 1 &\!\! 0 &\!\! 0 &\!\! 0 &\!\! 0 &\!\! 0 &\!\! 0 &\!\! 0 &\!\! 
         1 &\!\! 0  \\
        1 &\!\! 0 &\!\! 0 &\!\! 0 &\!\! 0 &\!\! 0 &\!\! 0 &\!\! 0 &\!\! 0 &\!\! 0 &\!\! 1 &\!\! 0 &\!\! 0 &\!\! 0 &\!\!
          0 &\!\! 0 &\!\! 0 &\!\! 0 &\!\! 1 &\!\! 0 &\!\! 0 &\!\! 0 &\!\! 0 &\!\! 0 &\!\! 0 &\!\! 0 &\!\! 0 &\!\! 
         0 &\!\! 1
\end{pmatrix}
} ,
\end{align*}
\end{proposition}

\begin{proof}
Direct computation using the methods implemented in Singular in \cite{khovhompaper}. 

Alternatively, one can use the expressions for~$A$ given on the right-hand sides of~\eqref{eq:Psums} together with the results on tensor products of permutation matrix factorisations in \cite{br0707.0922,cr1006.5609,DRCR}. 
\end{proof}

We observe that curiously the numbers of nonzero eigenvalues of the matrices $M^{(\mathrm{A}_{11})}$, $M^{(\mathrm{A}_{17})}$, and $M^{(\mathrm{A}_{29})}$ encoding the action of~$A$ on Grothendieck groups respectively are 6, 7, and 8: 
\begin{align*}
& M^{(\mathrm{A}_{11})} \text{ has eigenvalues } \big(3 - \sqrt{3}\big)^{\times 2}, \, 2^{\times 2}, \, \big(3 + \sqrt{3}\big)^{\times 2}
\, , 
\\
& M^{(\mathrm{A}_{17})} \text{ has approximate eigenvalues } 0.9358^{\times 2}, \,  3, \, 3.305^{\times 2},
7.759^{\times 2}
\, , 
\\
& M^{(\mathrm{A}_{29})} \text{ has approximate eigenvalues } 2.229^{\times 2} , \, 3.368^{\times 2}, \, 4.923^{\times 2}, \, 19.48^{\times 2}
\, . 
\end{align*}

\section{Calabi-Yau completions and orbifold equivalences}\label{sec:liftingOE}

In this section we 
	discuss under which assumptions 
an orbifold equivalence between two homologically smooth and proper dg algebras implies an orbifold equivalence between their Calabi-Yau completions. 
	If these assumptions hold, 
the ADE equivalences of Section~\ref{subsubsec:ADEorbi} lift to orbifold equivalences between Ginzburg algebras of Dynkin quivers. 

\medskip

To make sense of the concept of orbifold equivalence between dg algebras we need to organise the contents of Section~\ref{subsec:perfectdual} into a single bicategory which is studied in detail in \cite[App.\,A.2]{bfk1105.3177v3}: 

\begin{definition}\label{def:DGs}
The \textsl{bicategory of homologically smooth dg algebras} $\DGs_k$ has homologically smooth dg algebras over~$k$ as objects, and its categories of 1-morphisms $A\rightarrow B$ are those of perfect modules, namely~$\DGs_k(A,B) = \Perf( A^{\textrm{op}} \otimes_k B)$. 
Horizontal composition in $\DGs_k$ is given by the derived tensor product over the intermediate dg algebra, i.\,e.~for $X \in \DGs_k(A,B)$ and $Y \in \DGs_k(B,C)$ we have $Y\otimes X = Y \Lotimesl_B X$,
and the unit on~$A$ is given by $I_A = A$ viewed as an $A^{\textrm{op}} \otimes_k A$-module. 
\end{definition}

A dg Morita equivalence between homologically smooth dg algebras $A,B$ is precisely an equivalence $A \cong B$ in $\DGs_k$. Thus, since every equivalence has invertible quantum dimensions, Morita equivalence is a special kind of orbifold equivalence in $\DGs_k$.

\subsection[Adjunctions and quantum dimensions in $\DGsp_k$]{Adjunctions and quantum dimensions in $\boldsymbol{\DGsp_k}$}

To discuss adjoints, we restrict our considerations to homologically smooth dg algebras which are also proper. 
They form a full subbicategory $\DGsp_k \subset \DGs_k$ which is equivalent to the bicategory of saturated dg categories as shown in \cite[Prop.\,A.11\,\&\,A.12]{bfk1105.3177v3}. More importantly for us, though, we have the following immediate consequence of Proposition~\ref{prop:2.5}.

\begin{proposition}
$\DGsp_k$ has left and right adjoints.
\end{proposition}

In this subsection we shall present these adjunctions very explicitly in terms of K-projective resolutions, which will lead to simple expressions for quantum dimensions in $\DGsp_k$ in terms of Casimir elements. 

\medskip

For a $1$-morphism $X \in \DGsp_k (A,B)$ the left and right adjoints are given in terms of dualising complexes by 
$
\dX=X^{\vee} \otimes_B \Omega_B
$
and
$
X^{\dagger}=\Omega_A \otimes_A X^{\vee}
$, 
as in~\eqref{eqn:2.2}. From these expressions it is evident that generically $1$-morphisms in $\DGsp_k $ are not ambidextrous. 

\begin{lemma}
\label{lem:5.2}
Let $A,B \in \DGsp_k$ and let $X \in \DGsp_k (A,B)$ be ambidextrous. Then there are canonical isomorphisms
$$
\dX \otimes_B \Theta_B \cong \Theta_A \otimes_A \dX
\, , \quad 
X^{\dagger} \otimes_B \Theta_B \cong \Theta_A \otimes_A X^{\dagger} \, .
$$
\end{lemma} 

\begin{proof}
Putting together the isomorphisms given in \eqref{eqn:2.1} and \eqref{eqn:2.3}, we have
$$
{}^{\dagger}\! X \otimes_B \Theta_B \cong \Theta_A \otimes_A X^{\dagger} \, . 
$$
Since $X$ is assumed to be ambidextrous, the assertion follows.
\end{proof}

It will be convenient to have alternative expressions for the left and right adjoints of a $1$-morphism in $\DGsp_k $. This is the content of the next lemma.

\begin{lemma}
\label{lem:5.3}
Let $A,B \in \DGsp_k$ and $X \in \DGsp_k (A,B)$. Then there are canonical isomorphisms
\be\label{eq:XdualRhom}
{}^{\dagger}\! X \cong \RHom_{A^{\op}}(X,A)
\, , \quad 
X^{\dagger} \cong \RHom_B(X,B) \, .
\ee
\end{lemma}

\begin{proof}
By definition we know that ${}^{\dagger}\! X=X^{\vee} \otimes_B \Omega_B \cong X^{\vee} \Lotimesl_B (B^{\op})^*$ and $X^{\dagger}=\Omega_A \otimes_A X^{\vee} \cong (A^{\op})^* \Lotimesl_A X^{\vee}$. Therefore, in view of \cite[Lem.\,3.26]{Petit13}, we have
$$
{}^{\dagger}\! X \cong \RHom_{A^{\op}\otimes_k B}(X,A^{\op}\otimes_k B) \Lotimes_B (B^{\op})^* \cong \RHom_{A^{\op}}(X,A) \, ,
$$
and similarly
$$
X^{\dagger} \cong (A^{\op})^* \Lotimes_A \RHom_{A^{\op}\otimes_k B}(X,A^{\op}\otimes_k B) \cong \RHom_{B}(X,B) \, .
$$
This completes the proof.
\end{proof}

As a consequence of this result we obtain the following useful characterisation of an ambidextrous $1$-morphism in the bicategory $\DGsp_k $.

\begin{proposition}
\label{prop:5.4}
Let $A,B \in \DGsp_k$ and $X \in \DGsp_k (A,B)$. Then $X$ is ambidextrous if and only if there is an isomorphism
$$
X \otimes_A \Theta_A \cong \Theta_B \otimes_B X \, .
$$
\end{proposition}

\begin{proof}
We first show that $X$ is ambidextrous if and only if $(X^{\dagger})^{\dagger} \cong X$. If $X$ is ambidextrous then, by Lemma \ref{lem:5.3}, we have the following sequence of isomorphisms
\be\label{eq:Xdd}
(X^{\dagger})^{\dagger} \cong ({}^{\dagger}\! X)^{\dagger} \cong \RHom_A ({}^{\dagger}\! X, A) \cong \RHom_A (\RHom_{A^{\op}}(X,A),A) \, .
\ee
Since $B$ is assumed to be proper, $X$ is perfect as a dg $A^{\op}$-module, which implies that the right-hand side of~\eqref{eq:Xdd} is isomorphic to $X$. Conversely, if $(X^{\dagger})^{\dagger} \cong X$ then, applying Proposition~\ref{prop:2.5}, we get two pairs of adjoint functors $({}^{\dagger}\! X \Lotimesl_B - , X \Lotimesl_A -)$ and $(X^{\dagger}\Lotimesl_B -, X\Lotimesl_A-)$. Thus there is an isomorphism of functors
$$
{}^{\dagger}\! X \Lotimesl_B - \cong X^{\dagger}\Lotimesl_B - \, , 
$$
and since the functor determines the perfect dg $B^{\op} \otimes_k A$-module up to isomorphism, we have ${}^{\dagger}\! X \cong X^{\dagger}$ as perfect dg $B^{\op} \otimes_k A$-modules.

Let us now prove the statement. Using the isomorphism ${}^{\dagger}\! X \cong \Theta_A \otimes_A X^{\dagger} \otimes_B \Omega_B$ of~\eqref{eqn:2.3}, we find an isomorphism
$$
{}^{\dagger}\!(X^{\dagger}) \cong \Omega_B \otimes_B (X^{\dagger})^{\dagger} \otimes_A \Theta_A \, . 
$$
On the other hand, again by Lemma \ref{lem:5.3}, we have a sequence of isomorphisms
\be\label{eq:dXd}
{}^{\dagger}\!(X^{\dagger}) \cong \RHom_{B^{\op}}(X^{\dagger},B) \cong \RHom_{B^{\op}}(\RHom_B (X,B),B) \, .
\ee
Because $A$ is assumed to be proper, we have that $X$ is perfect as a dg $B$-module and therefore the right-hand side of~\eqref{eq:dXd} is isomorphic to $X$. The desired conclusion is now a consequence of the preceding remarks. 
\end{proof}

Our next goal is to isolate representatives in the homotopy category of the two pairs of  evaluation and coevaluation maps associated to a given ambidextrous $1$-morphism $X \in \DGsp_k (A,B)$. 
To start with, we shall always regard the canonical isomorphisms~\eqref{eq:XdualRhom} as identifications. Owing to Proposition \ref{prop:2.4}, there are canonical isomorphisms
\begin{align}
\nu_X &\colon X \Lotimes_A {}^{\dagger}\! X \xlongrightarrow{\cong}\RHom_{A^{\op}}(X,X) 
\, , \label{eq:nuX}
\\
\widetilde{\nu}_X&\colon X^{\dagger} \Lotimes_B X \xlongrightarrow{\cong}\RHom_{B}(X,X)
\, . \label{eq:nutildeX}
\end{align}
Following the remark after Proposition \ref{prop:2.4}, we obtain a representative in $\Kbf(B^{\op} \otimes_k B)$ and $\Kbf(A^{\op}\otimes_k A)$, respectively, for these canonical isomorphisms. 

To make this more explicit, we choose a K-projective resolution $\pi\colon P \to X$ in $\Dbf(A^{\op}\otimes_k B)$ and put ${}^{\dagger}\! P=\Hom_{A^{\op}}(P,A)$ and $P^{\dagger}=\Hom_B (P,B)$. Then~\eqref{eq:nuX} is represented by the map 
\begin{align}
\nu_P \colon P \otimes_A {}^{\dagger}\! P & \lra \Hom_{A^{\op}}(P,P) \, ,
\nonumber
\\
p \otimes f & \longmapsto \big( p' \mapsto p f(p') \big) \, . 
\label{eq:numap}
\end{align}
Correspondingly, \eqref{eq:nutildeX} is represented by the map 
\begin{align*}
\widetilde{\nu}_P \colon P^{\dagger} \otimes_B P & \lra \Hom_{B}(P,P) \, ,
\\
g \otimes p & \longmapsto \big( p' \mapsto (-1)^{\vert p' \vert \vert p \vert}g(p')p  \big) \, .
\end{align*}

The \textsl{Casimir elements} $\sum_i x_i \otimes {}^{\dagger} x_i \in P \otimes_A {}^{\dagger}\! P$ and $\sum_i y_i^{\dagger} \otimes y_i \in P^{\dagger} \otimes_B P$ are defined as the preimages of the identity under the isomorphisms $\nu_P$ and $\widetilde \nu_P$, respectively: 
$$
\sum_i x_i \otimes {}^{\dagger} x_i = \nu_P^{-1}( 1_P ) 
\, , \quad
\sum_i y_i^{\dagger} \otimes y_i = \widetilde \nu_P^{-1}( 1_P )  \, .
$$
With this preparation, we can now present explicit representatives for the adjunction maps 
\be\label{eq:adjX}
\begin{aligned}[c]
\ev_X&\colon {}^{\dagger}\! X \Lotimes_B X \lto A \, ,\\
\tev_X&\colon X \Lotimes_A X^{\dagger} \lto B \, ,
\end{aligned}
\qquad
\begin{aligned}[c]
\coev_X &\colon B \lto X \Lotimes_A {}^{\dagger}\! X \, ,\\
\tcoev_X &\colon A \lto X^{\dagger} \Lotimes_B X \, .\\
\end{aligned}
\ee

\begin{proposition}
\label{prop:4.5}
Let $X \in \DGsp_k (A,B)$ and let $\pi\colon P \to X$ be a K-projective resolution as above. Then the adjunction maps~\eqref{eq:adjX} are respectively represented in $\Kbf(A^{\op}\otimes_k A)$ and $\Kbf(B^{\op}\otimes_k B)$ by the maps 
$$
\begin{aligned}[c]
\varepsilon_P&\colon {}^{\dagger}\! P \otimes_B P \lto A \, ,\\
\widetilde{\varepsilon}_P&\colon P \otimes_A P^{\dagger} \lto B \, ,
\end{aligned}
\qquad
\begin{aligned}[c]
\eta_P &\colon B \lto P \otimes_A {}^{\dagger}\! P \, ,\\
\widetilde{\eta}_P &\colon A \lto P^{\dagger} \otimes_B P \, ,\\
\end{aligned}
$$
given by
$$
\begin{aligned}[c]
\varepsilon_P(f \otimes p)&= f(p) \, , \vphantom{\sum_i}\\
\widetilde{\varepsilon}_P(p \otimes g)&= (-1)^{\vert p\vert \vert g\vert} g(p) \, , \vphantom{\sum_i}
\end{aligned}
\qquad
\begin{aligned}[c]
\eta_P (b)&= \sum_i x_i \otimes {}^{\dagger} x_i b \, ,\\
\widetilde{\eta}_P(a) &=\sum_i y_i^{\dagger} \otimes y_i a \, . 
\end{aligned}
$$
\end{proposition}

\begin{proof}
We first recall how the evaluation and coevaluation maps associated to~$X$ are defined. As we are treating the 
isomorphisms of Lemma \ref{lem:5.3} as identifications, these maps are determined by the two pairs of adjoint functors $(-\Lotimesl_B X, \RHom_{A^{\op}}(X,-))$ and $(X \Lotimesl_A-,\RHom_B (X,-))$. To be more specific, one defines the evaluation maps $\ev_X$ and $\tev_X$ to be, respectively, the counits at $A$ and $B$ for the former and the latter adjunction. To define the coevaluation maps, one considers correspondingly the units at $B$ and $A$ for these pairs of adjoint functors, which we write as $\delta_X \colon B \to \RHom_{A^{\op}}(X, B\Lotimesl_B X)$ and $\widetilde{\delta}_X\colon A \to \RHom_B (X, X\Lotimesl_A A)$. One then defines $\coev_X$ and $\tcoev_X$ by
$$
\coev_X = \nu_X^{-1} \circ (\lambda_X \circ -)\circ \delta_X 
\, , \quad
\tcoev_X = \widetilde{\nu}_X^{-1} \circ (\rho_X \circ -)\circ \widetilde{\delta}_X \, .
$$

Now we proceed to verify the claim. As we have already remarked at the end of Section~\ref{subsect:2.3}, $P$ is a K-projective object of $\Dbf(A^{\op})$ and $\Dbf(B)$, forgetting the dg $B$-module and dg $A^{\op}$-module structure, respectively. Therefore we can apply the argument given in \cite[Sect.\,4.2]{Pauksztello08} to compute representatives in $\Kbf(A^{\op}\otimes_k A)$ for the evaluation map $\ev_X\colon {}^{\dagger}\! X \Lotimesl_B X \to A$ and the canonical map $\widetilde{\delta}_X\colon A \to \RHom_B (X, X\Lotimesl_A A)$. If we represent the former by $\varepsilon_P \colon {}^{\dagger}\! P \otimes_B P \to A$, as in the statement, and the latter by $\widetilde{\delta}_P\colon A \to \Hom_B (P, P \otimes_A A)$, then they are given by
$$
\varepsilon_P(f \otimes p)=f(p)
\, , \quad 
\big[ \widetilde{\delta}_P(a) \big] (p)=(-1)^{\vert a \vert \vert p \vert} p \otimes a \, ,
$$
respectively, for all $f \in {}^{\dagger}\! P$, $p \in P$ and $a \in A$. 

In an entirely analogous manner, one may compute representatives in $\Kbf(B^{\op}\otimes_k B)$ for the evaluation map $\tev_X\colon X \Lotimesl_A X^{\dagger} \to B$ and the canonical map $\delta_X \colon B \to \RHom_{A^{\op}}(X, B\Lotimesl_B X)$: if we denote the former by $\widetilde{\varepsilon}_P \colon P \otimes_A P^{\dagger} \to B$ and the latter by $\delta_P \colon B \to \Hom_{A^{\op}}(P,B \otimes_B P)$, one obtains
$$
\widetilde{\varepsilon}_P(p \otimes g)=(-1)^{\vert p \vert \vert g \vert}g(p)
\, , \quad 
\big[ \delta_P (b) \big] (p)=b \otimes p
$$
for all $g \in P^{\dagger}$, $p \in P$ and $b \in B$. 

Finally, the canonical maps $\lambda_X \colon B \Lotimesl_B X \to X$ and $\rho_X \colon X \Lotimesl_A A \to X$ are represented in $\Kbf(A^{\op}\otimes_k B)$ by the maps $\lambda_P \colon B \otimes_B P \to P$ and $\rho_P \colon P \otimes_A A \to P$ given by
$$
\lambda_P(b \otimes p)=bp 
\, , \quad 
\rho_P(p \otimes a)=pa
$$
for all $p \in P$, $a \in A$ and $b \in B$. 

These formulas taken together with our previous remarks imply that the coevaluation maps $\coev_X \colon B \to X \Lotimesl_A {}^{\dagger}\! X$ and $\tcoev_X \colon A \to X^{\dagger} \Lotimesl_B X$ are represented in $\Kbf(B^{\op}\otimes_k B)$ and $\Kbf(A^{\op}\otimes_k A)$ by the maps $\eta_P \colon B \to P \otimes_A {}^{\dagger}\! P$ and $\widetilde{\eta}_P\colon A \to X^{\dagger} \otimes_B X$ given by
\begin{align*}
\eta_P =\nu_P^{-1}\circ (\lambda_P \circ-) \circ \delta_P
\, , \quad
\widetilde{\eta}_P =\widetilde{\nu}_P^{-1}\circ (\rho_P \circ-) \circ \widetilde{\delta}_P \, .
\end{align*}
Evaluating these expressions at the unit elements $e_B\in B$ and $e_A\in A$, respectively, we get
$$
\eta_P(e_B)=\sum_i x_i \otimes {}^{\dagger} x_i 
\, ,\quad 
\widetilde{\eta}_P(e_A) =\sum_i y_i^{\dagger} \otimes y_i \, .
$$
The required assertion now follows from the fact that $\eta_P$ and $ \widetilde{\eta}_P$ are, respectively, maps of dg $B^{\op}\otimes_k B$-modules and of dg $A^{\op}\otimes_k A$-modules.
\end{proof}

As an aside it should be noted that, by the definition of the Casimir elements $\sum_i x_i \otimes {}^{\dagger} x_i$ and $\sum_i y_i^{\dagger} \otimes y_i$, we have that for any $p \in P$,
\be\label{eq:Casimirbasis}
p=\sum_i x_i {}^{\dagger} x_i(p)=\sum_i (-1)^{\vert y_i^{\dagger}\vert \vert p \vert} y_i^{\dagger}(p)y_i \, .
\ee
This means that $P$ as a dg $B$-module has a dual basis $\{y_i,y_i^{\dagger}\}$, while as a dg $A^{\op}$-module has a dual basis $\{x_i,{}^{\dagger} x_i \}$.

\medskip

We have now accumulated all the information necessary to provide formulas for quantum dimensions in $\DGsp_k$. 
Recall from Definition~\ref{def:qdims} that the left and right quantum dimensions of an ambidextrous 1-morphism~$X$ involve its adjunction maps as well as the isomorphism $\alpha_X \colon {}^{\dagger}\! X \to X^{\dagger}$. For a K-projective resolution $\pi \colon P \to X$ we can choose a representative $\alpha_P \colon {}^{\dagger}\! P \to P^{\dagger}$ of $\alpha_X$ in $\Kbf(B^{\op}\otimes_k A)$. 

\begin{proposition}
\label{prop:dimP}
Let $X \in \DGsp_k (A,B)$ be ambidextrous and let $\pi\colon P \to X$ be a K-projective resolution as above. 
Then the left and right quantum dimensions of~$X$ are represented by the maps $\diml(P) \in \End_{\Kbf(A^{\op}\otimes_k A)}(A)$ and $\dimr(P) \in \End_{\Kbf(B^{\op}\otimes_k B)}(B)$ given by left multiplication with distinguished elements: 
\begin{align*}
\diml(P) : a & \longmapsto \sum_i \big[\alpha_P^{-1}(y_i^{\dagger})\big](y_i) \cdot a \, , 
\\
\dimr(P) : b & \longmapsto \sum_i (-1)^{\vert{}^{\dagger} x_i \vert \vert x_i \vert} \big[\alpha_P ({}^{\dagger} x_i)\big](x_i) \cdot b \, .
\end{align*}
\end{proposition}

\begin{proof}
In the notation of Proposition \ref{prop:4.5}, the left quantum dimension of $X$ is represented in $\Kbf(A^{\op}\otimes_k A)$ by the map 
$$
\diml(P)=\varepsilon_P \circ (\alpha_P^{-1} \otimes 1_P) \circ \widetilde{\eta}_P \, .
$$
By the same token, the right quantum dimension of $X$ represented in $\Kbf(B^{\op}\otimes_k B)$ by the map 
$$
\dimr(P)=\widetilde{\varepsilon}_P \circ (1_P \otimes \alpha_P) \circ \eta_P \, .
$$
Using the explicit expressions given in Proposition \ref{prop:4.5}, this yields the required conclusion.
\end{proof}

\subsection{Lifting orbifold equivalences to Calabi-Yau completions}

In this subsection we 
	argue that if certain 1-morphisms are isomorphic in $\DGs_k$, then 
the construction of Calabi-Yau completion of a dg algebra is compatible with orbifold equivalences. 

\medskip

For the remainder of this section, we fix an integer~$n$ as well as two homologically smooth and proper dg algebras $A$ and $B$. 
Furthermore, we assume that $X \in \DGsp_k (A,B)$ is ambidextrous and exhibits an orbifold equivalence 
$$
A \sim B \, . 
$$ 

We denote the $n$-Calabi-Yau completions of~$A$ and~$B$ by 
$$
\Acal=\Pi_n(A) 
\, , \quad 
\Bcal=\Pi_n(B) \, , 
$$ 
and we further set\footnote{Note that 
	the 
tensor products in~\eqref{eq:Xcal} are also left derived tensor products since~$\Acal$ and~$\Bcal$ are, respectively, K-projective as a dg $A$-module and K-projective as a dg $B^{\op}$-module.}
\be\label{eq:Xcal}
\Xcal = X \otimes_A \Acal
\, , \quad 
\Xcal' = \Bcal \otimes_B X\, 
\qquad 
{}^{\dagger}\! \Xcal = \Acal \otimes_A {}^{\dagger}\! X
\, , \quad 
\Xcal'^{\dagger}= X^{\dagger} \otimes_B \Bcal \, .
\ee
Viewed as perfect dg $A^{\op} \otimes_k B$-modules, we note that there is a natural isomorphism $\varphi\colon \Xcal \to \Xcal'$.  
To see this, let us look at each of the summands of $\Acal$ and $\Bcal$ separately: First, we have natural isomorphisms
$$
X \otimes_A A \cong X \cong B \otimes_B X \, .
$$
On the other hand, by repeated application of Proposition \ref{prop:5.4} and recalling that $\theta_A=\Theta_A[n-1]$ and $\theta_B=\Theta_B[n-1]$, we find that for each $i \geqslant 1$ there is a natural isomorphism
$$
X \otimes_A \theta_A^{\otimes_A i} \cong \theta_B^{\otimes_B i} \otimes_B X \, , 
$$
hence $\Xcal \cong \Xcal'$ as dg $A^{\op} \otimes_k B$-modules. 
Similarly, we have ${}^{\dagger}\! \Xcal \cong \Xcal'^{\dagger}$ as dg $B^{\op} \otimes_k A$-modules. 

If $\Xcal$ and $\Xcal'$ satisfy the stronger condition of being isomorphic as dg $\Acal^{\op} \otimes_k \Bcal$-modules, and similarly for ${}^{\dagger}\! \Xcal$ and $\Xcal'^{\dagger}$, we will show that~$\Xcal$ exhibits an orbifold equivalence between~$\Acal$ and~$\Bcal$. 
We collect all our assumptions on the orbifold equivalence $A\sim B$ that we make throughout this section: 
\begin{assumption} 
\label{assump}
Let $X \in \DGsp_k (A,B)$ be ambidextrous and exhibit an orbifold equivalence $A\sim B$. 
Moreover, assume that there are isomorphisms $\Xcal \cong \Xcal'$ and ${}^{\dagger}\! \Xcal \cong \Xcal'^{\dagger}$ 
in $\DGs_k$. 
\end{assumption}

	The validity of Assumption \ref{assump} essentially depends on the possibility of lifting the isomorphism $\Xcal \cong \Xcal'$ as dg $A^{\op} \otimes_k B$-modules to an isomorphism $\Xcal \cong \Xcal'$ as dg $\Acal^{\op} \otimes_k \Bcal$-modules. One way one might tackle this problem is to consider $\Acal$ (respectively, $\Bcal$) as a dg $A$-ring (respectively, dg $B$-ring) in the sense of \cite{ww1412.4229} and think of the natural isomorphism $\varphi \colon \Xcal = X \otimes_A \Acal \to \Bcal \otimes_B X = \Xcal'$ as an `entwining structure' in the sense of \cite{BM98}. If this indeed can be done, then the required conditions on $\varphi$ will turn $ \Xcal'$ into a dg $\Acal^{\op} \otimes_k \Bcal$-module, and the same will be true for $\Xcal$ by considering the quasi-inverse of $\varphi$. What is more, $\varphi$ can be promoted to an isomorphism of dg $\Acal^{\op} \otimes_k \Bcal$-modules, as wished. 
	(Another way one might be able to address this problem is by means of homotopy transfer along $\varphi$. The problem here is that, \textsl{a priori}, it is not at all clear that the dg $\Bcal$-module structure on $\Xcal$ one gets by using this technique is compatible with its natural dg $\Acal^{\op}$-module structure, and the same can be said of $\Xcal'$. We expect this to work better at the level of $A_{\infty}$-algebras.) 

	We now proceed to (i) explicitly represent the adjunction maps of~$\Xcal$, and (ii) reduce the computation of quantum dimensions of~$\Xcal$ to the case of~$X$. 
	First we show that ${}^{\dagger} \!\Xcal$ and $\Xcal'^{\dagger}$
are respectively the left and right adjoints to $\Xcal$ and $\Xcal'$ in $\DGs_k $, as the notation suggests. 

\begin{lemma}
\label{lem:5.7}
There are canonical isomorphisms
$$
{}^{\dagger} \!\Xcal \cong \RHom_{\Acal^{\op}}(\Xcal,\Acal)
\, , \quad 
\Xcal'^{\dagger} \cong \RHom_{\Bcal}(\Xcal',\Bcal) \, . 
$$
\end{lemma}

\begin{proof}
We have a sequence of canonical isomorphisms
\begin{align*}
{}^{\dagger} \!\Xcal &= \Acal \otimes_A {}^{\dagger}\! X \\
 &\cong \Acal \Lotimes_A {}^{\dagger}\! X \\
 &\cong \Acal \Lotimes_A \RHom_{A^{\op}}(X,A) \\
 &\cong \RHom_{A^{\op}}(X,\Acal) \\
 &\cong \RHom_{A^{\op}}(X,\RHom_{\Acal^{\op}}(\Acal,\Acal)) \\
 &\cong \RHom_{\Acal^{\op}}(X \Lotimes_A \Acal, \Acal) \\
 &\cong \RHom_{\Acal^{\op}}(X \otimes_A \Acal, \Acal) \\
 &= \RHom_{\Acal^{\op}}(\Xcal,\Acal) \, .
 \end{align*}
 Here the second isomorphism is by virtue of Lemma \ref{lem:5.3}, the third one by Proposition \ref{prop:2.4}(i) and the fifth one is by virtue of the adjoint associativity law. In like manner, by making use of Proposition \ref{prop:2.4}(ii), one establishes the second isomorphism.
 \end{proof}

\begin{lemma}
\label{lem:5.8}
${}^{\dagger} \!\Xcal$ and $\Xcal'^{\dagger}$ are respectively the left and right adjoints to $\Xcal$ and $\Xcal'$ in $\DGs_k $. 
\end{lemma}

\begin{proof}
To begin with, $({}^{\dagger} \!\Xcal \Lotimesl_{\Bcal}-,\RHom_{\Acal}({}^{\dagger} \!\Xcal,-))$ is a pair of adjoint functors. Combining the fact that $\Xcal$ is perfect as a dg $\Acal^{\op}$-module with Lemma \ref{lem:5.7}, we get 
$$
\Xcal \cong \RHom_{\Acal}(\RHom_{\Acal^{\op}}(\Xcal,\Acal),\Acal)\cong \RHom_{\Acal}({}^{\dagger} \!\Xcal,\Acal) \, .
$$ 
Therefore, by the remark made after Proposition \ref{prop:2.4}, there is an isomorphism of functors 
$$
\Xcal \Lotimes_{\Acal} - \cong \RHom_{\Acal}({}^{\dagger} \!\Xcal,\Acal) \Lotimes_{\Acal} - \cong \RHom_{\Acal}({}^{\dagger} \!\Xcal,\Acal) \, .
$$
Hence we conclude that the functor ${}^{\dagger} \!\Xcal \Lotimesl_{\Bcal}-$ is left adjoint to $\Xcal \Lotimesl_{\Acal} -$, and consequently ${}^{\dagger} \!\Xcal$ is the left adjoint to $\Xcal$ in $\DGs_k $.

For the statement regarding $\Xcal'^{\dagger}$, we consider the pair of adjoint functors $(\Xcal' \Lotimesl_{\Acal}-,\RHom_{\Bcal}(\Xcal',-))$. Using again the remark following Proposition \ref{prop:2.4} and Lemma \ref{lem:5.7}, we have 
$$
\RHom_{\Bcal}(\Xcal',-)\cong \RHom_{\Bcal}(\Xcal',\Bcal)\Lotimes_{\Bcal}-\cong \Xcal'^{\dagger}\Lotimes_{\Bcal}- \, .
$$
It then follows that the functor $\Xcal'^{\dagger}\Lotimes_{\Bcal}-$ is right adjoint to $\Xcal' \Lotimesl_{\Acal}-$, and hence $\Xcal'^{\dagger}$ is the right adjoint to $\Xcal'$ in $\DGs_k $. 
\end{proof}

On the basis of this result, it is convenient to set 
$$
\Xcal^{\dagger} = \RHom_{\Bcal}(\Xcal,\Bcal) \, .
$$
Then, 
	due to Assumption~\ref{assump}, 
there is an isomorphism $\Xcal^{\dagger} \cong \Xcal'^{\dagger}$, so that $\Xcal^{\dagger}$ is right adjoint to $\Xcal$ in $\DGs_k$, we have ${}^{\dagger} \!\Xcal \cong \Xcal^{\dagger}$. 

We summarise our discussion so far as follows: 

\begin{proposition}
\label{prop:Xcaladjoint}
With the notation as above, $\Xcal \in \DGs_k (\Acal,\Bcal)$ is an ambidextrous $1$-morphism where ${}^{\dagger} \!\Xcal$ and $\Xcal^{\dagger}$ are its left and right adjoint, respectively.
\end{proposition}

In light of the above canonical isomorphisms let us from hereon identify 
$$
\dXcal = \RHom_{\Acal^{\op}} (\Xcal, \Acal) 
\, , \quad 
\Xcal^\dagger = \RHom_\Bcal(\Xcal,\Bcal) \, . 
$$
The (co)evaluation maps exhibiting $\dXcal, \Xcal^\dagger$ as adjoints of~$\Xcal$ are determined by the adjoint pairs $(-\Lotimesl_\Bcal \Xcal, \RHom_{\Acal^{\op}}(\Xcal,-))$ and $(\Xcal \Lotimesl_\Acal-,\RHom_\Bcal (\Xcal,-))$. 
As in the previous case of adjoints of~$X$, we shall now represent the adjunction maps for~$\Xcal$ with the help of a K-projective resolution $\pi: P \rightarrow X$. With this we set
\begin{align*}
\Pcal & = P \otimes_A \Acal 
\, , \quad 
\Pcal' = \Bcal \otimes_B P 
\, , 
\\
\dPcal & = \Hom_{\Acal^{\op}}(\Pcal, \Acal) 
\, , \quad
\Pcal'^\dagger = \Hom_\Bcal(\Pcal', \Bcal) 
\, , \quad 
\Pcal^\dagger = \Hom_\Bcal (\Pcal, \Bcal) 
\end{align*}
and note that 
	Assumption~\ref{assump} 
and Proposition~\ref{prop:Xcaladjoint} provide us with canonical isomorphisms
\begin{align*}
\varphi: \Pcal = P \otimes_A \Acal & \xlongrightarrow{\cong} \Bcal \otimes_B P = \Pcal' \, , 
\\
\psi: \Acal \otimes_A \dP & \xlongrightarrow{\cong} \dP \otimes_B \Bcal \, ,
\\
\widetilde\psi : \Acal \otimes_A P^\dagger & \xlongrightarrow{\cong} P^\dagger \otimes_B \Bcal \, . 
\end{align*}
Furthermore, we will employ the following canonical isomorphisms: 
$$
\begin{aligned}[c]
\beta :  \Acal \otimes_A \dP & \xlongrightarrow{\cong} \Hom_{A^{\op}}(P,\Acal) \, ,\\
u \otimes f & \longmapsto \big( p \mapsto u f(p) \big) \, ,
\end{aligned}
\qquad
\begin{aligned}[c]
\widetilde\beta :  P^\dagger \otimes_B \Bcal & \xlongrightarrow{\cong} \Hom_B(P,\Bcal) \, ,\\
g \otimes v & \longmapsto \big( p \mapsto (-1)^{|p| |v|}g(p) v \big) \, ,
\end{aligned}
$$
and 
$$
\begin{aligned}[c]
&\gamma :  \Hom_{A^{\op}}(P,\Acal) \xlongrightarrow{\cong} \dPcal \, ,\\
&F  \longmapsto \big( p \otimes u \mapsto F(p) u \big) \, ,
\end{aligned}
\qquad
\begin{aligned}[c]
&\widetilde\gamma :  \Hom_B(P,\Bcal) \xlongrightarrow{\cong} \Pcal'^\dagger \, ,\\
&G  \longmapsto \big( p \otimes u \mapsto (-1)^{|G| |v|} v G(p) \big) \, .
\end{aligned}
$$

Now we can state the analogue of Proposition~\ref{prop:4.5} for~$\Xcal$. 

\begin{proposition}
In the above notation the adjunction maps for~$\Xcal$ are represented in $\Kbf(\Acal^{\op}\otimes_k \Acal)$ and $\Kbf(\Bcal^{\op}\otimes_k \Bcal)$ by the maps 
$$
\begin{aligned}[c]
\varepsilon_\Pcal&\colon {}^{\dagger}\! \Pcal \otimes_\Bcal \Pcal \lto \Acal \, ,\\
\widetilde{\varepsilon}_\Pcal&\colon \Pcal \otimes_\Acal \Pcal^{\dagger} \lto \Bcal \, ,
\end{aligned}
\qquad
\begin{aligned}[c]
\eta_\Pcal &\colon \Bcal \lto \Pcal \otimes_\Acal {}^{\dagger}\! \Pcal \, ,\\
\widetilde{\eta}_\Pcal &\colon \Acal \lto \Pcal^{\dagger} \otimes_\Bcal \Pcal \\
\end{aligned}
$$
given by
$$
\begin{aligned}[c]
\varepsilon_\Pcal(F \otimes q)&= F(q) \, , \vphantom{\sum_i }
\\
\widetilde{\varepsilon}_\Pcal(q \otimes G)&= (-1)^{\vert q\vert \vert G\vert} G(q) \, , \vphantom{\sum_i }
\end{aligned}
\qquad
\begin{aligned}[c]
\eta_\Pcal (v)&= \sum_i ( x_i \otimes e_\Acal ) \otimes \gamma\beta ( e_\Acal \otimes {^\dagger x_i} ) v  \, ,\\
\widetilde{\eta}_\Pcal (u) &=\sum_i \Big( \widetilde\gamma \widetilde\beta (y_i^\dagger \otimes e_\Bcal) \circ \varphi \Big) \otimes (y_i \otimes e_\Acal) u \\
\end{aligned}
$$
where $e_\Acal$ and $e_\Bcal$ are the unit elements of~$\Acal$ and~$\Bcal$, respectively. 
\end{proposition}

\begin{proof}
It is clear that $\varepsilon_\Pcal, \widetilde\varepsilon_\Pcal$ represent the evaluation maps for~$\Xcal$, so we only have to consider $\eta_\Pcal, \widetilde\eta_\Pcal$. 
As an aside we observe that it follows from their expressions that the Casimir elements for~$\Pcal$ are $\sum_i ( x_i \otimes e_\Acal ) \otimes \gamma\beta ( e_\Acal \otimes {^\dagger x_i} )$ and $\sum_i ( \widetilde\gamma \widetilde\beta (y_i^\dagger \otimes e_\Bcal) \circ \varphi ) \otimes (y_i \otimes e_\Acal)$. 

We first examine the map $\eta_\Pcal: \Bcal \rightarrow \Pcal \otimes_\Acal {}^{\dagger}\! \Pcal $ in detail. 
It is induced by $\eta_P$ via a sequence of canonical isomorphisms:  
$$
\begin{tikzpicture}[
			     baseline=(current bounding box.base), 
			     >=stealth,
			     descr/.style={fill=white,inner sep=3.5pt}, 
			     normal line/.style={->}
			     ] 
\matrix (m) [matrix of math nodes, row sep=0.6em, column sep=2.5em, text height=1.5ex, text depth=0.9ex] {%
\Bcal && \Bcal \otimes_B B 
\\
\phantom{\B} && \Bcal \otimes_B P \otimes_A \dP
\\
\phantom{\B} && P \otimes_A \Acal \otimes_A \dP 
\\
\phantom{\B} && P \otimes_A \Hom_{A^{\op}}(P,\Acal) 
\\
\phantom{\B} && P \otimes_A \dPcal 
\\
\phantom{\B} && \Pcal \otimes_\Acal \dPcal \, . 
\\
};
\path[font=\small] (m-1-1) edge[->] node[above] {$ \cong $} (m-1-3);
\path[font=\small] (m-2-1) edge[->] node[above] {$ 1_\Bcal\otimes\eta_P $} (m-2-3);
\path[font=\small] (m-3-1) edge[->] node[above] {$ \varphi^{-1} \otimes 1_{\dP} $} (m-3-3);
\path[font=\small] (m-4-1) edge[->] node[above] {$ 1_P \otimes\beta $} (m-4-3);
\path[font=\small] (m-5-1) edge[->] node[above] {$ 1_P \otimes \gamma $} (m-5-3);
\path[font=\small] (m-6-1) edge[->] node[above] {$ \cong $} (m-6-3);
\end{tikzpicture}
$$
Note that to arrive from here at the expression for $\eta_\Pcal$ in the statement of the proposition one uses $\varphi(p\otimes e_\Acal) = e_\Bcal \otimes p$, which follows from the fact that~$\varphi$ is a bimodule map. 

To prove that $\eta_\Pcal$ really represents the coevaluation we have to verify that it is the preimage of the identity under the isomorphism 
\begin{align*}
\nu_\Pcal \colon \Pcal \otimes_\Acal \dPcal & \lra \Hom_{\Acal^{\op}}(\Pcal, \Pcal) \, ,
\\
q \otimes F & \longmapsto \big( q' \mapsto q F(q') \big) 
\end{align*}
as in~\eqref{eq:numap}. This is straightforward: 
\begin{align*}
\nu_\Pcal &\Big( \sum_i (x_i \otimes e_\Acal ) \otimes \gamma\beta (e_\Acal \otimes {}^\dagger x_i) \Big)(p \otimes u)
\\
& = \sum_i (x_i \otimes e_\Acal) \big( \gamma\beta (e_\Acal \otimes {}^\dagger x_i) \big) (p \otimes u)
\\
& = \sum_i (x_i \otimes e_\Acal) \big( \beta (e_\Acal \otimes {}^\dagger x_i) \big) (p)u
\\
& = \sum_i (x_i \otimes e_\Acal) {}^\dagger x_i(p) u
\\
& = \sum_i x_i \otimes {}^\dagger x_i(p) u
\\
& = \Big( \sum_i x_i  {}^\dagger x_i(p) \Big) \otimes u
\\
& = p \otimes u
\end{align*}
where in the third last step we used that $e_\Acal$ is actually the unit element of~$A$, and in the last step we used~\eqref{eq:Casimirbasis}. 

Similarly, the map $\widetilde\eta_\Pcal$ is given by the sequence of maps 
$$
\begin{tikzpicture}[
			     baseline=(current bounding box.base), 
			     >=stealth,
			     descr/.style={fill=white,inner sep=3.5pt}, 
			     normal line/.style={->}
			     ] 
\matrix (m) [matrix of math nodes, row sep=0.6em, column sep=2.5em, text height=1.5ex, text depth=0.9ex] {%
\Acal && \Acal \otimes_A A 
\\
\phantom{\B} && \Acal \otimes_A P^\dagger \otimes_B P
\\
\phantom{\B} && P^\dagger \otimes_B \Bcal \otimes_B P 
\\
\phantom{\B} && \Hom_{B}(P,\Bcal) \otimes_B P
\\
\phantom{\B} && \Pcal'^\dagger \otimes_B P
\\
\phantom{\B} && \Pcal'^\dagger \otimes_{\Bcal} \Pcal'
\\
\phantom{\B} && \Pcal^\dagger \otimes_{\Bcal} \Pcal \, . 
\\
};
\path[font=\small] (m-1-1) edge[->] node[above] {$ \cong $} (m-1-3);
\path[font=\small] (m-2-1) edge[->] node[above] {$ 1_\Acal \otimes \widetilde\eta_P $} (m-2-3);
\path[font=\small] (m-3-1) edge[->] node[above] {$ \widetilde\psi \otimes 1_{P} $} (m-3-3);
\path[font=\small] (m-4-1) edge[->] node[above] {$ \widetilde\beta \otimes 1_P $} (m-4-3);
\path[font=\small] (m-5-1) edge[->] node[above] {$ \widetilde\gamma \otimes 1_P $} (m-5-3);
\path[font=\small] (m-6-1) edge[->] node[above] {$ \cong $} (m-6-3);
\path[font=\small] (m-7-1) edge[->] node[above] {$ ( - \circ \varphi ) \otimes \varphi^{-1} $} (m-7-3);
\end{tikzpicture}
$$
A computation analogous to the one for $\eta_\Pcal$ shows that $\widetilde\eta_\Pcal$ is the preimage of the identity under the isomorphism
\begin{align*}
\widetilde\nu_\Pcal \colon \Pcal^\dagger \otimes_\Bcal \Pcal & \lra \Hom_{\Bcal}(\Pcal, \Pcal) \, ,
\\
F \otimes q & \longmapsto \big( q' \mapsto (-1)^{|q| |q'|} F(q') q \big) \, .
\end{align*}
This completes the proof. 
\end{proof}

With its adjunctions now under explicit control we can proceed to compute the quantum dimensions of $\Xcal \in \DGs_k(\Acal,\Bcal)$. Recall that in Proposition~\ref{prop:dimP} we computed the quantum dimensions of $X \in \DGsp_k(A,B)$ as those of its K-projective resolution. It turns out that the case of~$\Xcal$ can basically be reduced to that result: 

\begin{proposition}
Let $X \in \DGsp_k (A,B)$ be ambidextrous, let $\pi\colon P \to X$ be a K-projective resolution and set $\Pcal = P \otimes_A \Acal$ as above. 
Then the quantum dimensions of $\Xcal = X \otimes_A \Acal$ are represented by the maps $\diml(\Pcal) \in \End_{\Kbf(\Acal^{\op}\otimes_k \Acal)}(\Acal)$ and $\dimr(\Pcal) \in \End_{\Kbf(\Bcal^{\op}\otimes_k \Bcal)}(\Bcal)$ given by
\begin{align*}
\diml(\Pcal) : u & \longmapsto \big[\! \diml(P) \big](e_A) \cdot u \, , 
\\
\dimr(\Pcal) : v & \longmapsto \big[\! \dimr(P) \big](e_B) \cdot v \, .
\end{align*}
\end{proposition}

\begin{proof}
The quantum dimensions of~$\Pcal$ are by definition
$$
\diml(\Pcal)=\varepsilon_\Pcal \circ (\alpha_\Pcal^{-1} \otimes 1_\Pcal) \circ \widetilde{\eta}_\Pcal
 \, , \quad
\dimr(\Pcal)=\widetilde{\varepsilon}_\Pcal \circ (1_\Pcal \otimes \alpha_\Pcal) \circ \eta_\Pcal 
$$
where $\alpha_\Pcal: \dPcal \rightarrow \Pcal^\dagger$ is the isomorphism induced by $\alpha_P: \dP \rightarrow P^\dagger$: 
$$
\begin{tikzpicture}[
			     baseline=(current bounding box.base), 
			     >=stealth,
			     descr/.style={fill=white,inner sep=3.5pt}, 
			     normal line/.style={->}
			     ] 
\matrix (m) [matrix of math nodes, row sep=0.6em, column sep=2.5em, text height=1.5ex, text depth=0.9ex] {%
\dPcal && \Hom_{A^{\op}}(P, \Acal) 
\\
\phantom{\dPcal} && \Acal \otimes_A \dP 
\\
\phantom{\dPcal} && \dP \otimes \Bcal
\\
\phantom{\dPcal} && P^\dagger \otimes_B \Bcal
\\
\phantom{\dPcal} && \Hom_B(P,\Bcal) 
\\
\phantom{\dPcal} && \Pcal'^\dagger
\\
\phantom{\dPcal} && \Pcal^\dagger \, . 
\\
};
\path[font=\small] (m-1-1) edge[->] node[above] {$ \gamma^{-1} $} (m-1-3);
\path[font=\small] (m-2-1) edge[->] node[above] {$ \beta^{-1} $} (m-2-3);
\path[font=\small] (m-3-1) edge[->] node[above] {$ \psi $} (m-3-3);
\path[font=\small] (m-4-1) edge[->] node[above] {$ \alpha_P \otimes 1_\Bcal $} (m-4-3);
\path[font=\small] (m-5-1) edge[->] node[above] {$ \widetilde\beta $} (m-5-3);
\path[font=\small] (m-6-1) edge[->] node[above] {$ \widetilde\gamma $} (m-6-3);
\path[font=\small] (m-7-1) edge[->] node[above] {$ -\circ\varphi $} (m-7-3);
\end{tikzpicture}
$$
Hence we have 
\begin{align*}
\alpha_\Pcal & = (-\circ\varphi) \circ \widetilde\gamma \circ \widetilde\beta \circ (\alpha_P \otimes 1_\Bcal) \circ \psi \circ \beta^{-1} \circ \gamma^{-1} \, , 
\\
\alpha_\Pcal^{-1} & = \gamma \circ \beta \circ \psi^{-1} \circ (\alpha_P^{-1} \otimes 1_\Bcal) \circ \widetilde\beta^{-1} \circ \widetilde\gamma^{-1} \circ (- \circ \varphi^{-1})
\end{align*}
and we can compute
\begin{align*}
\big[\! \diml(P) \big](u) 
& = \varepsilon_\Pcal \Big( (\alpha_\Pcal^{-1} \otimes 1_\Pcal ) \big( \widetilde\eta_\Pcal(u) \big) \Big)
\\
& = \sum_i \varepsilon_\Pcal \Big( (\alpha_\Pcal^{-1} \otimes 1_\Pcal ) \Big(\big( \widetilde\gamma\widetilde\beta (y_i \otimes e_\Bcal) \circ \varphi \big) \otimes (y_i \otimes e_\Acal) u \Big)\Big)
\\
& = \sum_i \varepsilon_\Pcal \Big( \alpha_\Pcal^{-1} \big( \widetilde\gamma\widetilde\beta (y_i \otimes e_\Bcal) \circ \varphi \big) \otimes (y_i \otimes e_\Acal) \Big) \cdot u
\\
& = \sum_i \varepsilon_\Pcal \Big( \gamma\beta \big( e_\Acal \otimes \alpha_P^{-1}(y_i^\dagger) \big) \otimes (y_i \otimes e_\Acal) \Big) \cdot u
\\
& = \sum_i  \Big[ \gamma\beta \big( e_\Acal \otimes \alpha_P^{-1}(y_i^\dagger) \big)  \Big] (y_i \otimes e_\Acal) \cdot u
\\
& = \sum_i  \Big[ \beta \big( e_\Acal \otimes \alpha_P^{-1}(y_i^\dagger) \big)  \Big] (y_i) \cdot u
\\
& = \sum_i  \big[ \alpha_P^{-1}(y_i^\dagger) \big] (y_i) \cdot u
\\ 
& = \big[\! \diml(P) \big](e_A) \cdot u \, . 
\end{align*}
The computation of $\dimr(\Pcal)$ is analogous. 
\end{proof}

As an immediate corollary we find that if the quantum dimensions of~$X$ are invertible, then those of~$\Xcal$ are invertible too. 

\begin{theorem}
\label{thm:maintheorem}
Let $X \in \DGsp_k(A,B)$ exhibit an 
	orbifold equivalence $A\sim B$ as in Assumption~\ref{assump}. 
Then $\Xcal = X \otimes_A \Acal \in \DGs_k(\Acal,\Bcal)$ exhibits an orbifold equivalence $\Acal \sim \Bcal$ between the $n$-Calabi-Yau completions $\Acal = \Pi_n(A) $ and $\Bcal = \Pi_n(B)$. 
\end{theorem}

\subsection{Dynkin quivers and Ginzburg algebras}
\label{subsec:liftingOEDynkin}

	We would like to 
lift Theorem~\ref{thm:ADEOE} to the level of Ginzburg algebras. 
	Thanks to 
the relation between simple singularities and the derived representation theory of ADE type Dynkin quivers established in \cite{kst0511155}
	this is possible if Assumption~\ref{assump} is satisfied. 

\medskip

To review this relation let $W^{(\Gamma)}$ be a simple singularity as in~\eqref{eq:simplesing}, 
and let $Q^{(\Gamma)}$ be a Dynkin quiver of the same ADE type as $W^{(\Gamma)}$.\footnote{$Q^{(\Gamma)}$ is obtained from the corresponding Dynkin diagram~$\Gamma$ in Figure~\ref{fig:DynkinDiagrams} by choosing arbitrary orientations for its edges.} 
Furthermore, we denote by 
$$
\Tcal^{(\Gamma)}=\bigoplus_{i=1}^{\vert \Gamma \vert} T_i^{(\Gamma)}
$$ 
the direct sum of the simple objects $T_i^{(\Gamma)}$ in $\hmfgr ( \C[x,y,z], W^{(\Gamma)} )$. 
Using Corollary~3.16 of \cite{kst0511155} and the remark that follows it, we have: 

\begin{theorem}
$\Tcal^{(\Gamma)}$ is a tilting object in $\hmfgr ( \C[x,y,z], W^{(\Gamma)} )$. Moreover, the path algebra $\C Q^{(\Gamma)}$ is isomorphic to the endomorphism algebra of $\Tcal^{(\Gamma)}$.
\end{theorem}

In particular, there is an equivalence of triangulated categories
$$
 \Hom \!\big(\Tcal^{(\Gamma)},-\big) \colon \hmfgr \!\big( \C[x,y,z], W^{(\Gamma)} \big) 
 \longrightarrow 
 \Dbfb \big(\!\mod(\C Q^{(\Gamma)})\big)
$$
between the homotopy category of matrix factorisations of $W^{(\Gamma)}$ and the bounded derived category of finitely generated left modules over the path algebra $\C Q^{(\Gamma)}$. As a consequence we can restate the result on orbifold equivalences in Theorem~\ref{thm:ADEOE}. We will do so in terms of the bicategory $\mathcal{H}_{\C}$ whose objects are hereditary $\C$-algebras of finite representation type and whose 1-morphism categories $\mathcal{H}_{\C}(A,B)$ are the categories of complexes of projective left modules over $A^{\op}\otimes_{\C} B$. 

By a well-know result due to Gabriel, all objects of $\mathcal{H}_{\C}$ are isomorphic to path algebras $\C Q^{(\Gamma)}$ of Dynkin quivers $Q^{(\Gamma)}$ of some ADE type~$\Gamma$. Since these path algebras are homologcially smooth and proper when viewed as dg algebras concentrated in degree $0$ (with zero differential), we see that $\mathcal{H}_{\C}$ is a full subbicategory of the bicategory $\DGsp_{\C}$. 

Let $X$ be the $1$-morphism in $\LGgr$ exhibiting the orbifold equivalence between $\C Q^{(\Gamma)}$ and $\C Q^{(\Gamma')}$ in Theorem~\ref{thm:ADEOE}. By an argument analogous to that found in \cite[Prop.\,3.16]{d0904.4713}, one can show that the diagram of functors
$$
\begin{tikzpicture}[
			     baseline=(current bounding box.base), 
			     >=stealth,
			     descr/.style={fill=white,inner sep=3.5pt}, 
			     normal line/.style={->}
			     ] 
\matrix (m) [matrix of math nodes, row sep=4em, column sep=2.5em, text height=1.5ex, text depth=0.9ex] {%
\hmfgr \!\big(\C[x,y,z], W^{(\Gamma)} \big) &&&&& \hmfgr \!\big( \C[x',y',z'], W^{(\Gamma')} \big)
\\
\Dbfb \big(\!\mod(\C Q^{(\Gamma)})\big) &&&&& \Dbfb \big(\!\mod(\C Q^{(\Gamma')})\big)
\\
};
\path[font=\small] (m-1-1) edge[->] node[auto] {$ X \otimes - $} (m-1-6);
\path[font=\small] (m-1-1) edge[->] node[right] {$ \Hom(\Tcal^{(\Gamma)},-) $} (m-2-1);
\path[font=\small] (m-1-6) edge[->] node[left] {$ \Hom(\Tcal^{(\Gamma')},-) $} (m-2-6);
\path[font=\small] (m-2-1) edge[->] node[below] {$ \Hom(\Tcal^{(\Gamma)\vee}\otimes_{\C} \Tcal^{(\Gamma')}, X) \Lotimesl_{\C Q^{(\Gamma)}}- $} (m-2-6);
\end{tikzpicture}
$$
commutes up to a natural isomorphism. Coupling this fact with Proposition~\ref{prop:2.3}, it follows that 
$$
\Hom_{\hmfgr( \C[x,y,z,x',y',z'], W^{(\Gamma)} - W^{(\Gamma')})} \big(\Tcal^{(\Gamma)\vee}\otimes_{\C} \Tcal^{(\Gamma')}, X \big)
$$ 
is a complex of projective $(\C Q^{(\Gamma)})^{\op}\otimes_{\C} \C Q^{(\Gamma')}$-modules and therefore a $1$-morphism in $\mathcal{H}_{\C}$. Moreover, it is ambidextrous and has invertible quantum dimensions because both $\Hom(\Tcal^{(\Gamma)},-)$ and $\Hom(\Tcal^{(\Gamma')},-)$ are equivalences of triangulated categories. 
Hence Theorem~\ref{thm:ADEOE} can be rephrased as follows: 

\begin{proposition}
In $\mathcal{H}_{\C}$ (and thus in $\DGsp_{\C}$), there are orbifold equivalences 
\begin{align*}
& \C Q^{(\mathrm{A}_{2d-1})} \sim \C Q^{(\mathrm{D}_{d+1})}
\, , \quad
\\
&\C Q^{(\mathrm{A}_{11})} \sim \C Q^{(\mathrm{E}_6)} 
\, , \quad
\C Q^{(\mathrm{A}_{17})} \sim \C Q^{(\mathrm{E}_7)} 
\, , \quad
\C Q^{(\mathrm{A}_{29})} \sim \C Q^{(\mathrm{E}_8)} \, .
\end{align*}
\end{proposition}

Invoking Theorems~\ref{thm:maintheorem} and~\ref{thm:2.8} we 
	find: 

\begin{corollary}\label{cor:GinzburgDynkinOE}
	For those integers $n\geqslant 2$ for which the orbifold equivalences above satisfy Assumption~\ref{assump}, 
there are orbifold equivalences 
\begin{align*}
& \Gamma_n \big(Q^{(\mathrm{A}_{2d-1})}\big) \sim \Gamma_n\big( Q^{(\mathrm{D}_{d+1})}\big)
\, , \quad
\\
&\Gamma_n\big( Q^{(\mathrm{A}_{11})} \big)\sim \Gamma_n\big(Q^{(\mathrm{E}_6)} \big)
\, , \quad
\Gamma_n\big(Q^{(\mathrm{A}_{17})}\big) \sim\Gamma_n\big( Q^{(\mathrm{E}_7)} \big)
\, , \quad
\Gamma_n\big( Q^{(\mathrm{A}_{29})}\big) \sim\Gamma_n\big( Q^{(\mathrm{E}_8)}\big) 
\end{align*}
between Ginzburg algebras 
	in $\DGs_{\C}$. 
\end{corollary}

\begin{remark}\label{rem:Fukaya}
There is a link between derived categories of Ginzburg algebras for $n=3$ and Fukaya categories of quasi-projective $3$-folds associated to meromorphic quadratic differentials. Here we briefly describe this link following \cite{Smith13}, and offer some comments on the relation to our results. 

The input data is a marked bordered surface $(S,M)$ with $\partial S \neq \emptyset$ arising from a meromorphic quadratic differential $\phi$ on a Riemann surface~$C$ with a non-empty set of poles of order $\geqslant 3$.\footnote{What this means, in more detail, is that~$S$ is obtained as the real blow-up of~$C$ at poles of $\phi$ of order $\geqslant 3$, while $M \subset S$ is given by the poles of $\phi$ of order $\leqslant 2$ together with the distinguished tangent directions at the poles of order $\geqslant 3$.} To each ideal triangulation $T$ of $S$ with vertices at $M$ there is an associated quiver with superpotential $(Q_T, W_T)$, as originally described in \cite{LF09}, and we may consider the corresponding Ginzburg algebra $\Gamma_3(Q_T, W_T)$ (defined in \cite[Def.\,5.1.1]{Ginzburg06} and \cite[Sect.\,6.2]{Keller11}). 

On the other hand, to the meromorphic quadratic differential~$\phi$ one can also associate a quasi-projective $3$-fold $Y_{\phi}$, which is an affine conic fibration over $C$ with nodal fibres over the zeroes of $\phi$, singular fibres at infinity over the double poles, and empty fibres over higher order poles. It is further shown in \cite{Smith13} that for each $\beta \in H^2(Y_{\phi},\Z_2)$ there is a well-defined Fukaya category $\mathcal{F}(Y_{\phi}, \beta)$. 

This construction is inspired by rank-two Gaiotto theories originally introduced in \cite{Gaiotto07} and further studied in \cite{ACCERV12}. The main result of \cite{Smith13} asserts that for a certain class $\beta_0 \in H^2(Y_{\phi},\Z_2)$ there is a fully faithful embedding
\be\label{eq:derivedintoFukaya}
\Dbfb \big(\Gamma_3(Q_T, W_T)\big) 
\longhookrightarrow 
\Dbfb \big(\mathcal{F}(Y_{\phi}, \beta_0)\big) 
\ee
where $\Dbfb (\mathcal{F}(Y_{\phi},\beta_0))$ denotes the derived category of $\mathcal{F}(Y_{\phi}, \beta_0)$, defined as the degree zero cohomology of the category of twisted complexes over the idempotent completion of $\mathcal{F}(Y_{\phi}, \beta_0)$. 
It is further conjectured in \cite{Smith13} that if $\mathcal{K}(Y_{\phi}, \beta_0)$ denotes the full $A_{\infty}$-subcategory of $\mathcal{F}(Y_{\phi}, \beta_0)$ generated by Lagrangian matching spheres, then there is an equivalence of triangulated categories  
$$
\Dbfb \big(\Gamma_3(Q_T, W_T)\big) 
\xlongrightarrow{\cong} 
\Dbfb \big(\mathcal{K}(Y_{\phi}, \beta_0)\big) \, .
$$

As pointed out in \cite{BS15}, the Dynkin quivers $Q^{(\mathrm{A}_{d-1})}$ and $Q^{(\mathrm{D}_{d+1})}$ can be obtained from an ideal triangulation of two marked bordered surfaces $(S,M)$ as above. To be more specific, the Dynkin quiver $Q^{(\mathrm{A}_{d-1})}$ is obtained from an ideal triangulation of an unpunctured disc with $d+2$ points on its boundary, while the Dynkin quiver $Q^{(\mathrm{D}_{d+1})}$ is obtained from an ideal triangulation of a once-punctured disc with $d+1$ points on its boundary; the meromorphic quadratic differentials $\phi^{(\mathrm{A}_{d-1})}$ and $\phi^{(\mathrm{D}_{d+1})}$ corresponding to these bordered surfaces are also explicitly described in \cite{BS15}. 

On the other hand, the quasi-projective $3$-fold $Y_{\phi^{(\mathrm{A}_{d-1})}}$ associated to the meromorphic differential $\phi^{(\mathrm{A}_{d-1})}$ has been studied before, notably by Khovanov and Seidel in \cite{KS02}. It is described as the total space of a Lefschetz fibration over the complex plane with affine quadric fibres and $d$ nodal singular fibres. To be more explicit, $Y_{\phi^{(\mathrm{A}_{d-1})}}$ can be realised as a hypersurface $\{x^2 + y^2 + z^2 + t^{d}=1 \}$ in $\C^4$, where the fibration over the complex plane comes from the projection to the $t$-variable. In this case, it is known that for each $\beta^{(\mathrm{A}_{d-1})} \in H^2(Y_{\phi^{(\mathrm{A}_{d-1})}},\ZZ_2)$ the Lagrangian matching spheres generate the Fukaya category $\mathcal{F}(Y_{\phi^{(\mathrm{A}_{d-1})}},\beta^{(\mathrm{A}_{d-1})})$, and that for a certain class $\beta_0^{(\mathrm{A}_{d-1})} \in H^2(Y_{\phi^{(\mathrm{A}_{d-1})}},\ZZ_2)$ there is an equivalence of triangulated categories
$$
\Dbfb\left(\Gamma_3(Q^{(\mathrm{A}_{d-1})})\right) \xlongrightarrow{\cong} \Dbfb\big(\mathcal{F}(Y_{\phi^{(\mathrm{A}_{d-1})}},\beta_0^{(\mathrm{A}_{d-1})})\big) \, .
$$
It should be remarked, though, that we are not aware of one single reference for these assertions.\footnote{For the generation statement, one may appeal to \cite{AS12}, where, in a broader setup, it is shown that all Lagrangian spheres are generated by Lagrangian matching spheres. As for the second statement, \cite{ST01} or \cite{KS02} show that the relevant Floer algebra is intrinsically formal, from which one may conclude that it should be quasi-isomorphic to the Ginzburg algebra $\Gamma_3(Q^{(\mathrm{A}_{d-1})})$, as the latter has zero superpotential. We thank Ivan Smith for explaining this to us.}

As regards to the quasi-projective $3$-fold $Y_{\phi^{(\mathrm{D}_{d+1})}}$ associated to the meromorphic differential $\phi^{(\mathrm{D}_{d+1})}$, it was pointed out to us by I.~Smith that it may be obtained as the total space of a Lefschetz fibration over the complex plane with $d$ nodal singular fibres and a reducible fibre at the origin, which can also be described as a hypersurface $\{x^2 + t(y^2 + z^2) + t^{d}=1 \}$ in $\C^4$, where again the fibration over the complex plane comes from the projection to the $t$-variable. It seems conceivable that in this case one can also argue, as above, that for each $\beta^{(\mathrm{D}_{d+1})} \in H^2(Y_{\phi^{(\mathrm{D}_{d+1})}},\ZZ_2)$ the Lagrangian matching spheres generate the Fukaya category $\mathcal{F}(Y_{\phi^{(\mathrm{D}_{d+1})}},\beta^{(\mathrm{D}_{d+1})})$, and that for some class $\beta_0^{(\mathrm{D}_{d+1})} \in H^2(Y_{\phi^{(\mathrm{D}_{d+1})}},\ZZ_2)$ there is conjecturally an equivalence of triangulated categories
$$
\Dbfb\left(\Gamma_3(Q^{(\mathrm{D}_{d+1})})\right) \xlongrightarrow{\cong} \Dbfb\big(\mathcal{F}(Y_{\phi^{(\mathrm{D}_{d+1})}},\beta_0^{(\mathrm{D}_{d+1})})\big) \, .
$$

In view of this, and taking note of Corollary~\ref{cor:GinzburgDynkinOE}, it is natural for us to conjecture that there are orbifold equivalences 
$$
Y_{\phi^{(\mathrm{A}_{2d-1})}} \sim Y_{\phi^{(\mathrm{D}_{d+1})}} \, . 
$$
Here the relevant bicategory should be that of symplectic manifolds and Lagrangian correspondences studied in \cite{ww0708.2851}. We also conjecture a similar connection between the quasi-projectice $3$-folds (which are real symplectic 6-folds) associated to quivers of Dynkin type $\mathrm{A}_{11},\mathrm{A}_{17},\mathrm{A}_{29}$ and those associated to quivers of type $\mathrm{E}_6,\mathrm{E}_7,\mathrm{E}_8$, respectively. 
We stress though that an interpretation of the derived category of E-type Ginzburg algebras in terms of Fukaya categories of quasi-projective $3$-folds is unknown to us. What is known, however, is that E-type Dynkin quivers cannot be obtained from ideal triangulations of marked bordered surfaces as in type A and D, cf.~\cite[Thm.\,13.3\,\&\,Rem.13.5]{FST}. 
\end{remark}

\appendix

\section{Differential graded algebras and their modules}\label{app:dgalgebras}

A \textsl{differential graded algebra} over $k$ (or simply a \textsl{dg algebra}) is a graded associative $k$-algebra $A=\bigoplus_{i \in \ZZ} A^i$ equipped with a $k$-linear map $d=d_A \colon A \to A$ of degree $+1$ with $d^2=0$, such that the Leibniz rule
$$
d(a_1 a_2)= (d a_1)a_2+(-1)^{\vert a_1 \vert } a_1 (d a_2) \, ,
$$
holds for any homogenous $a_1 \in A$ and all $a_2 \in A$ (the degree of a homogenous element $a$ being denoted by $\vert a \vert$). A \textsl{map of dg algebras} $\varphi :A \to B$ is a homomorphism of graded algebras with $\varphi \circ d_A=d_B \circ \varphi$. 

Any dg algebra $A$ over $k$ gives rise to a complex
$$
\cdots \longrightarrow A^{i-1}  \xlongrightarrow{d^{i-1}} A^i  \xlongrightarrow{d^{i}} A^{i+1} \longrightarrow \cdots
$$
where the differentials are given by $d^{i}=d\vert_{A^i}$. The cohomology of this complex is denoted by $\Hrm(A)$ and is simply called the \textsl{cohomology of the dg algebra} $A$. Note that $\Hrm(A)$ is a graded  $k$-algebra.

The \textsl{tensor product} $A \otimes_k B$ of two dg algebras is their tensor product as graded  $k$-algebras, with the differential given by
$$
d(a \otimes b)=(d_A a) \otimes b+(-1)^{\vert a\vert} a \otimes (d_B b) \, .
$$
We also assign to each dg algebra $A$ an \textsl{opposite dg algebra} $A^{\op}$. This is defined to be the dg algebra whose underlying $k$-vector space is the same as for $A$ and whose multiplication is given by $a \cdot b=(-1)^{\vert a\vert \vert b\vert}b a$ for homogeneous $a,b \in A$.

Let $A$ be a dg algebra over $k$. A \textsl{differential graded left $A$-module} (or simply a \textsl{dg left $A$-module}) is a graded left $A$-module $M=\bigoplus_{i \in \ZZ}M^i$ equipped with a $k$-linear map $\nabla=\nabla_M \colon M \to M$ of degree $+1$ with $\nabla^2=0$, such that 
$$
\nabla (a m)=(da)m+(-1)^{\vert a\vert} a (\nabla m),
$$
holds for any homogenous $a \in A$ and all $m \in M$. One may similarly define the notion of a dg right $A$-module.

It is worth noticing that any dg right $A$-module $M$ can be seen as a dg left $A^{\op}$-module via $a \cdot m=(-1)^{\vert a \vert \vert m \vert}m a$ for $a \in A$ and $m \in X$. From now on, when we say ``$M$ is a dg $A$-module'' we will mean ``$M$ is a dg left $A$-module''. In a similar manner, ``$M$ is a dg $A^{\op}$-module'' will mean ``$M$ is a dg right $A$-module''.

Any ordinary $k$-algebra $A$ can be considered as a dg algebra with $A^0=A$ and $A^{i}=0$ for all $i \neq 0$. In this case, a dg $A$-module is simply a complex of left $A$-modules. In general, like a dg algebra $A$, a dg $A$-module $M$ is also inherently a complex. We denote by $\Hrm(M)$ the cohomology of $M$ and note that it has a natural structure of a graded left $\Hrm(A)$-module.

If $M$ and $N$ are dg $A$-modules, a \textsl{map of dg modules} $f \colon M \to N$ is just a homomorphism of graded modules, of degree $0$, such that $f \circ \nabla_M=\nabla_N \circ f$. With these morphisms, dg $A$-modules form a category in which submodules and quotient modules, kernels, images, coimages and cokernels are defined as usual.  

On this category we define bifunctors $\Hom_A(-,-)$ and $-\otimes_A-$. For dg $A$-modules $M$ and $N$, a \textsl{graded map of degree} $n$ is a $k$-linear map $f \colon M \to N$ such that
$$
f(a m)=(-1)^{\vert a \vert n} a f(m) \, ,
$$
for any homogeneous $a \in A$ and all $m \in M$. In other words, $f$ is a homomorphism from $M$ to $N$, regarded just as modules over the graded algebra $A$. The set of all such $f$ is a $k$-vector space which we denote by $\Hom^n_A(M,N)$. The graded $k$-vector space $\Hom_A(M,N)=\bigoplus_{n \in \ZZ} \Hom^n_A(M,N)$ has a differential defined for each $f \in \Hom^n_A(M,N)$ as
$$
\nabla f=\nabla_N \circ f - (-1)^{n} f \circ \nabla_M \, .
$$
Thus $\Hom_A (M,N)$ is a dg $k$-vector space. Note that the set of maps of dg modules from $M$ to $N$ is the $k$-vector space of cycles of degree $0$ in the complex $\Hom_A (M,N)$.

Let now~$M$ be a dg $A^{\op}$-module and~$N$ a dg $A$-module. Considered just as modules over the graded algebra $A$, they define a graded $k$-vector space $M \otimes_A N$ which becomes a dg $k$-vector space when the differential~$\nabla$ is defined by
$$
\nabla (m \otimes n)= (\nabla_M m) \otimes n + (-1)^{\vert m \vert} m \otimes (\nabla_N n)
$$
for any homogeneous $m \in M$ and all $n \in N$. 

Let $A$ and $B$ be two dg algebras. A \textsl{dg $B$-$A$-bimodule} $X$ is a $k$-vector space which is a dg $A^{\op}$-module and a dg $B$-module and satisfies the associativity condition $b(x a)=(b x)a$ for all $x \in X$, $a \in A$ and $b \in B$. Every dg $A$-module~$M$ is a dg $A$-$k$-bimodule, where the structure of the right $k$-module on $M$ is given by $\lambda m=m \lambda=m(\lambda 1_A)$ for $\lambda \in k$ and $m \in M$. Similarly, every dg $A^{\op}$-module~$N$ is a dg $k$-$A$-bimodule, where the structure of the left $k$-module on $N$ is given by $n \lambda=\lambda n=(\lambda 1_A)n$ for $\lambda \in k$ and $n \in N$. In particular, any dg algebra $A$ is a dg $A$-$A$-bimodule. It is also worthwhile to mention that, by setting $(a \otimes b) x=(-1)^{(\vert x\vert+\vert b\vert) \vert a\vert}b (xa)$, each dg $B$-$A$-bimodule may be regarded as a dg $A^{\op} \otimes_k B$-module, or vice versa. This reduction carries with it the definitions of $\Hom$ and $\otimes$ for dg bimodules. Here and in the sequel we shall always assume that such an identification has been made.

Let $X$ be a dg $A^{\op} \otimes_k B$-module, $M$ be a dg $A^{\op}$-module and $N$ be a dg $B$-module. Then the graded $k$-vector spaces $\Hom_{A^{\op}}(X,M)$ and $\Hom_{B}(X,N)$ have a dg $B^{\op}$-module and a dg $A$-module structure, respectively, with operation of elements of $B$ and $A$ such that
$$
(b f)(x)=(-1)^{(\vert f \vert +\vert x\vert) \vert b \vert}f(xb)
\, , \quad 
(g a)(x)=g(a x)
$$
for all $a \in A$, $b \in B$, $x \in X$, $f \in \Hom_A(X,M)$ and $g \in \Hom_{B^{\op}}(X,N)$. Over and above that, a dg $A^{\op} \otimes_k C$-module $Y$ and a dg $C^{\op} \otimes_k B$-module $Z$ make the graded $k$-vector spaces $\Hom_{A^{\op}}(X,Y)$ and $\Hom_{B}(X,Z)$ a dg $B^{\op} \otimes_k C$-module and a dg $C^{\op} \otimes_k A$-module, respectively, by the formulas
$$
(b f c)(x)=(-1)^{\vert f \vert \vert b \vert + (\vert b\vert +\vert c\vert)\vert x\vert} f(xb)c
\, , \quad 
(c g a)(x)=c g (a x)
$$
for all $a \in A$, $b \in B$, $c \in C$, $x \in X$, $f \in \Hom_A(X,Y)$ and $g \in \Hom_{B^{\op}}(X,Z)$.

Similarly, let $X$ be a dg $A^{\op} \otimes_k B$-module, $M$ be a dg $A$-module and $N$ be a dg $B^{\op}$-module. The tensor products $X \otimes_A M$ and $N \otimes_B X$ then have a dg $B$-module and a dg $A^{\op}$-module structure, respectively, given by
$$
(m \otimes x) b=m \otimes (xb)
\, , \quad 
a(x \otimes n)=(ax) \otimes n
$$
for all $m \in M$, $n \in N$, $x \in X$, $a \in A$ and $b \in B$. Moreover, a dg $A^{\op} \otimes_k C$-module $X$ and a dg $C^{\op} \otimes_k B$-module $Y$ make $Y \otimes_C X$ a dg $A^{\op} \otimes_k B$-module by setting
$$
a(x \otimes y)b=(ax) \otimes (yb)
$$
for all $x \in X$, $y \in Y$, $a \in A$ and $b \in B$. 

$\Hom$ and $\otimes$ are related by the ``adjoint associativity law''. This statement means, as usual, that if $A$ and $B$ are dg algebras, $M$ is a dg $A$-module, $N$ is a dg $B$-module, $P$ is a dg $A^{\op}$-module, $Q$ is a dg $B^{\op}$-module and $X$ is a dg $A^{\op} \otimes_k B$-module, there are natural isomorphisms of graded $k$-vector spaces
\begin{align*}
\Hom_{B} ( X \otimes_A M, N) &\cong \Hom_{A} (M, \Hom_{B}(X,N)) \, , 
\\
\Hom_{A^{\op}} ( Q \otimes_B X, P) &\cong \Hom_{B^{\op}} (Q, \Hom_{A^{\op}}(X,P))
\end{align*}
which are functorial in $M$, $N$, $P$, $Q$ and $X$.  Expressed in a different way, the tensor functors $X \otimes_A -$ and $- \otimes_B X$ are, respectively, the left adjoints of the functors $\Hom_B(X,-)$ and $\Hom_A^{\op}(X,-)$.

\medskip

Suppose that~$A$ and~$B$ are dg algebras and let $\varphi \colon A \to B$ be a map of dg algebras. We equip~$B$ with both a dg $A$-module and a dg $A^{\op}$-module structure as follows. The dg $A$-module structure on~$B$ is defined by the action $a \cdot b=\varphi(a) b$ for $a \in A$ and $b \in B$, and the dg $A^{\op}$-module structure on~$B$ is given by $b \cdot a=b \varphi(a)$. 

As a direct consequence of the definitions involved, we have the following result which we use to compute quantum dimensions in Section~\ref{sec:liftingOE}. 

\begin{proposition}
\label{prop:2.1}
Let $A$, $A'$, $B$, $B'$ be dg algebras, let $\varphi\colon A \to A'$ and $\psi\colon B \to B'$ be maps of dg algebras, and let~$X$ be a dg $A^{\op} \otimes_k B$-module. 
\begin{enumerate}
\item[(i)]For a dg $A^{\op}$-module~$M$, the $k$-linear map 
\begin{align*}
\xi \colon \Hom_{A^{\op}} (X,M) & \lra \Hom_{A'^{\op}} ( X \otimes_A A'  , M) \, , 
\\
f & \longmapsto \big( x \otimes a' \mapsto f(x) a' \big)
\end{align*}
is an isomorphism of dg $B^{\op}$-modules.
\item[(ii)]For a dg $B$-module $N$, the $k$-linear map 
\begin{align*}
\widetilde{\xi} \colon \Hom_B (X,N) & \lra \Hom_{B'} (B' \otimes_B X, N) \, , 
\\
g & \longmapsto  \big( b' \otimes x \mapsto (-1)^{\vert g \vert \vert b' \vert} b' g(x) \big)
\end{align*}
is an isomorphism of dg $A$-modules.
\end{enumerate}
\end{proposition}

\end{document}